\documentclass{amsart}

\usepackage[foot]{amsaddr}
\usepackage{tikz}
\usetikzlibrary{decorations.pathreplacing}
\usepackage{graphicx,url,etoolbox}
\usepackage{lipsum}
\usepackage{dsfont}
\usepackage{amsmath,amssymb}
\usepackage{verbatim}
\usepackage{graphicx,cancel,xcolor,comment,graphicx,geometry}
\usepackage{comment}
\usepackage{xparse}
\usepackage{expl3}
\usepackage{blindtext, rotating}
\usepackage{caption}
\usepackage{subcaption}
\usepackage{tikz}

\RequirePackage[bookmarks,%
colorlinks,%
urlcolor=blue,%
citecolor=red,%
linkcolor=darkspringgreen,%
linkcolor=blue,%
hyperfigures,%
pagebackref,
pdfcreator=LaTeX,%
breaklinks=true,%
pdfpagelayout=SinglePage,%
bookmarksopen=true,%
bookmarksopenlevel=2]{hyperref}

\let\oldtocsection=\tocsection
\let\oldtocsubsection=\tocsubsection 
\let\oldtocsubsubsection=\tocsubsubsection

\renewcommand{\tocsection}[2]{\vspace{0.5em}\hspace{0em}\oldtocsection{#1}{#2}}
\renewcommand{\tocsubsection}[2]{\vspace{0.5em}\hspace{1em}\oldtocsubsection{#1}{#2}}
\renewcommand{\tocsubsubsection}[2]{\vspace{0.5em}\hspace{2em}\oldtocsubsubsection{#1}{#2}}

\theoremstyle{definition}

\numberwithin{equation}{section}

\renewenvironment{proof}{{\bfseries \noindent Proof.}}{\demo}
\newcommand\xqed[1]{%
	\leavevmode\unskip\penalty9999 \hbox{}\nobreak\hfill
	\quad\hbox{#1}}
\newcommand\demo{\xqed{$\square$}}

\def\R{\mathbb R}

\def\N{\mathbb N}

\def\C{\mathbb C}
\def\HH{\mathcal H}

\def\AA{\mathcal A}

\def\la {{\lambda}}

\newcommand {\nc}   {\newcommand}
\nc {\be}   {\begin{equation}} \nc {\ee}   {\end{equation}} \nc
{\beq}  {\begin{eqnarray}} \nc {\eeq}  {\end{eqnarray}} \nc {\beqs}
{\begin{eqnarray*}} \nc {\eeqs} {\end{eqnarray*}}
\def\edc{\end{document}}

\providecommand{\abs}[1]{\lvert#1\rvert}
\providecommand{\norm}[1]{\lVert#1\rVert}

\numberwithin{equation}{section}

\newtheorem{Thm}{Theorem}[section]
\newtheorem{lem}{Lemma}[section]
\newtheorem{prop}{Proposition}[section]

\newtheorem{rk}{Remark}[section]
\newtheorem{assump}{Assumption}

\definecolor{carnelian}{rgb}{0.7, 0.11, 0.11}
\definecolor{carmine}{rgb}{0.59, 0.0, 0.09}
\definecolor{burgundy}{rgb}{0.5, 0.0, 0.13}
\definecolor{darkmidnightblue}{rgb}{0.0, 0.2, 0.4}
\definecolor{dimgray}{rgb}{0.75, 0.75, 0.75}
\definecolor{palecarmine}{rgb}{0.69, 0.25, 0.21}
 
\numberwithin{dummy}{section}

\numberwithin{equation}{section}

\def\AA{\mathcal A}
\def\HH{\mathbf{\mathcal H}}

\providecommand{\norm}[1]{\lVert#1\rVert}
\providecommand{\abs}[1]{\lvert#1\rvert}

\ExplSyntaxOn
\NewDocumentCommand{\biglittlecap}{m}
{
\sheljohn_biglittecap:nn { #1 }
}
\tl_new:N \l__sheljohn_biglittecap_input_tl
\cs_new_protected:Npn \sheljohn_biglittecap:nn #1
{
\tl_set:Nn \l__sheljohn_biglittecap_input_tl { #1 }
\regex_replace_all:nnN
{ ([A-Z]+ | \cC.\{?[A-Z]+\}?) }
{ \cB\{\c{large}\1\cE\} }   
\l__sheljohn_biglittecap_input_tl
\tl_use:N \MakeUppercase{\l__sheljohn_biglittecap_input_tl}
}
\ExplSyntaxOff
\providecommand{\keywords}[1]{\textbf{\textit{Index terms---}} #1}

\begin{document}

\title[Stability of wave-plate transmission problem]{Stability for an interface transmission problem of  wave-plate equations with dynamical boundary controls}
\author{Zahraa Abdallah$^{1,2}$}
\author{Stéphane Gerbi$^{2}$}
\author{Chiraz Kassem$^{1}$}
\author{Ali Wehbe$^{1}$}
\address{$^{1}$Lebanese University, Faculty of Sciences Khawarizmi Laboratory of Mathematics and Applications-KALMA Hadath-Beirut, Lebanon.}
\address{$^{2}$Laboratoire de Mathématiques-LAMA UMR 5127-CNRS and Université Savoie Mont Blanc, Campus Scientifique
73376 Le Bourget-du-Lac Cedex, France.}

\email{abdallazahraa02@gmail.com, stephane.gerbi@univ-smb.fr, shiraz.kassem@hotmail.com, ali.wehbi@ul.edu.lb}

\setcounter{equation}{0}
\begin{abstract}
We investigate a two-dimensional transmission model consisting of a wave equation and a Kirchhoff plate equation with dynamical boundary controls under geometric conditions. The two equations are coupled through transmission conditions along a steady interface between the domains in which the wave and plate equations evolve, respectively. Our primary concern is the stability analysis of the system, which has not appeared in the literature. For this aim, using a unique continuation theorem, the strong stability of the system is proved without any geometric condition and in the absence of compactness of the resolvent. Then, we show that our system lacks exponential (uniform) stability. However, we establish a polynomial energy decay estimate of type $1/t$ for smooth initial data using the frequency domain approach from semigroup theory, which combines a contradiction argument with the multiplier technique. This method leads to certain geometrical conditions concerning the wave's and the plate's domains.
\end{abstract}
\maketitle
\pagenumbering{roman}
\maketitle
\tableofcontents
\clearpage
\pagenumbering{arabic}
\setcounter{page}{1}
\newpage

\section{Introduction}
\subsection{Presentation of the System}\label{Systemwaveplate}
  In this paper, we investigate a wave-plate system coupled through transmission conditions along a steady interface between the domains in which the wave and plate equations evolve respectively with dynamical boundary controls. More precisely, let $\Omega \subset \R^2$ be an open bounded set with  Lipschitz boundary $\partial \Omega$ such that $\overline{\Omega}=\overline{\Omega}_1 \cup \overline{\Omega}_2$, where $\Omega_i$ is a bounded set with Lipschitz boundary $\partial \Omega_i,$ $i=1,2$ such that $\Omega_1 \cap \Omega_2 =\emptyset$. We denote by $I$ the interior of $\overline{\Omega}_1 \cap \overline{\Omega}_2,$ which is called the interface between $\Omega_1$ and $\Omega_2,$ and $\Gamma_i =\partial \Omega_i \backslash \overline{I}$ represents the exterior boundary of $\Omega_i$ with positive measure, $i=1,2$. We consider the following interface transmission wave-plate model:
\begin{equation}\label{p5-sys2}
\left\{ \begin{array}{llll}
\displaystyle  u_{tt}-\Delta u =0, & \text{in} \ \  \Omega_1 \times \R_+^*,\vspace{0.15cm}\\
\displaystyle w_{tt}+\Delta^2 w =0, & \text{in} \ \   \Omega_2 \times \R_+^*,\vspace{0.15cm}\\
\displaystyle u=w, \ \ \mathcal{B}_1 w=0, \ \ \mathcal{B}_2w=\partial_{\nu_1}u, & \text{on} \ \ I \times \R_+^*,\vspace{0.15cm}\\
\displaystyle \partial_{\nu_1} u + \eta =0, & \text{on} \ \ \Gamma_1 \times \R_+^*,\vspace{0.15cm}\\ 
\displaystyle \mathcal{B}_1 w+ \xi =0, & \text{on} \ \ \Gamma_2 \times  \R_+^*\vspace{0.15cm},\\
\displaystyle \mathcal{B}_2 w- \zeta =0, & \text{on} \ \ \Gamma_2 \times \R_+^*,
\end{array}
\right.
\end{equation}
with the initial conditions:
\begin{equation} \label{initial1}
 u(x,0)=u_0(x), \ \ u_t(x,0)=u_1(x), \ \ x \in \Omega_1, \ \ \eta(x,0)=\eta_0(x), \ \ x \in  \Gamma_1,
\end{equation}
\begin{equation}\label{initial2} w(x,0)=w_0(x), \ \ w_t(x,0)=w_1(x), \ \  x \in  \Omega_2, \ \ \xi(x,0)=\xi_0(x), \ \ \zeta(x,0)=\zeta_0(x), \ \ x \in  \Gamma_2,
\end{equation}
where $\eta, \ \xi$ and $\zeta$ denote the dynamical boundary controls, $\nu_i =(\nu_{i1},\nu_{i2})$ is the unit outward normal vector along  $\partial \Omega_i$, and $\tau_i =(-\nu_{i2},\nu_{i1})$ is the unit tangent vector along $\partial \Omega_i$, $i=1,2$. Now, let's proceed with defining the gradient of a scalar function $f(x_1, x_2)$ as follows: $$\nabla f =\left( f_{x_1},f_{x_2}\right),$$ where $f_{x_1}$ and $f_{x_2}$ represent the partial derivatives of $f$ with respect to $x_1$ and $x_2$, respectively. Next, the Laplacian operator is defined as the divergence of the gradient of $f,$ given by $$\Delta f= \text{div} \left(\nabla f \right)= f_{x_1 x_1}+f_{x_2 x_2},$$ where $f_{x_1 x_1}$ and $f_{x_2 x_2}$ are the second partial derivatives of $f$ with respect to $x_1$ and $x_2$. Additionally, in this problem's context, we introduce the constant $\mu \in \left(0,\frac{1}{2}\right)$ as the Poisson coefficient. Moreover, we define the boundary operators $\mathcal{B}_1$ and $\mathcal{B}_2$ on $\partial \Omega_2$ as follows:
$$
\mathcal{B}_1 f=\Delta f+(1-\mu) \mathcal{C}_1f
$$
and
$$
\mathcal{B}_2 f=\partial_{\nu_2}\Delta f+(1-\mu)\partial_{\tau_2}\mathcal{C}_2 f,
$$
where
$$
\mathcal{C}_1 f=2\nu_{21}\nu_{22}f_{x_1x_2} -\nu_{21}^2f_{x_2x_2}-\nu_{22}^2 f_{x_1x_1} \ \ \ \text{and} \ \ \ \mathcal{C}_2 f= (\nu_{21}^2-\nu_{22}^2)f_{x_1x_2}-\nu_{21}\nu_{22}\left(f_{x_1x_1}-f_{x_2x_2} \right).
$$
The damping of the system is made via the indirect damping mechanism that involves the following first-order differential equations: 
\begin{equation}\label{p5-diff}
\left\{	\begin{array}{llll}
\displaystyle \eta_t - u_t + \eta  =0, & \text{on} \ \ \Gamma_1 \times \R_+^*,\vspace{0.15cm}\\ 
\xi_t -\partial_{\nu_2} w_t + \xi  =0, & \text{on} \ \ \Gamma_2 \times \R_+^*\vspace{0.15cm},\\
\displaystyle \zeta_t -w_t+ \zeta =0, & \text{on} \ \ \Gamma_2 \times \R_+^*.
\end{array}
\right.
\end{equation}
Moreover, by trivial computation, we know that 
\begin{equation}\label{p5-JEL}
	\mathcal{C}_1 f=-\partial^2_{\tau_2}f-\partial_{\tau_2} \nu_{22}f_{x_1} +\partial_{\tau_2} \nu_{21}f_{x_2} \ \ \ \text{and} \ \ \ \mathcal{C}_2 f=\partial_{\nu_2 \tau_2} f-\partial_{\tau_2}\nu_{21}f_{x_1}-\partial_{\tau_2}\nu_{22}f_{x_2}.
\end{equation}

Let $u, \eta, w, \xi$ and $\zeta$ be smooth solutions of system \eqref{p5-sys2}-\eqref{p5-diff}. We define their associated energy by
$$
E(t)=\frac{1}{2} \left\{\int_{\Omega_1}\left(|\nabla u|^2+|u_t|^2  \right)dx+ \int_{\Gamma_1} |\eta|^2 d \Gamma \right\}+ \frac{1}{2} \left\{a(w,\overline{w})+\int_{\Omega_2}|w_t|^2dx + \int_{\Gamma_2} \left(|\xi|^2+|\zeta|^2\right) d \Gamma\right\} ,
$$
where the sesquilinear form $a:H^2(\Omega_2)\times H^2(\Omega_2) \longmapsto \C$ is defined by 
\begin{equation}\label{p5-1.3}
a(f,\overline{g})=\int_{\Omega_2} \left[f_{x_1x_1}\overline{g}_{x_1x_1}+f_{x_2x_2}\overline{g}_{x_2x_2}+\mu\left(f_{x_1x_1} \overline{g}_{x_2x_2}+f_{x_2x_2} \overline{g}_{x_1x_1}\right) +2(1-\mu)f_{x_1x_2}\overline{g}_{x_1x_2}\right]dx.
\end{equation}
For further analysis, the following Green's formula (see \cite{lagnese1989boundary}) is used in the study:
\begin{equation}\label{GF}
a(f,\overline{g})=\int_{\Omega_2} \Delta^2 f \overline{g} dx +\int_{\partial \Omega_2 }\left( \mathcal{B}_1f \partial_{\nu_2}\overline{g}-\mathcal{B}_2 f \overline{g}   \right) d\Gamma, \ \ \forall f\in H^4 (\Omega_2 ),  \ g \in H^2 (\Omega_2).
\end{equation}
\begin{lem}\label{p5-lemenergy}
Let  $U=(u,u_t,\eta ,w,w_t,\xi,\zeta )$ be a regular solution of  system \eqref{p5-sys2}-\eqref{p5-diff}. Then, the energy $E(t)$  satisfies the following estimation 
$$
E^\prime (t)=- \int_{\Gamma_1} |\eta|^2d\Gamma-\int_{\Gamma_2} \left(|\xi|^2 + |\zeta|^2\right) d\Gamma \leq 0.
$$
\end{lem}
\begin{proof}
	Firstly, multiplying the first and second equations of \eqref{p5-sys2}  by $\overline{u}_t$ and $\overline{w}_t$ respectively, then using Green's formula in $\Omega_1$ and $\Omega_2$ and taking the real part, we get 
	\begin{equation}\label{p5-lemenergy1}
	\frac{1}{2} \frac{d}{dt}\left\{\int_{\Omega_1}\left(|\nabla u|^2+|u_t|^2  \right)dx+a(w,\overline{w})+\int_{\Omega_2}|w_t|^2dx\right\} + \Re \left\{ \int_{\Gamma_1} \eta \overline{u}_t d\Gamma + \int_{\Gamma_2}\left(\xi\partial_{\nu_2} \overline{w}_t+\zeta\overline{w}_t\right) d \Gamma\right\}=0 ,
	\end{equation}
	where $\Re$ stands for the real part of a complex number. Secondly, by inserting the equations \eqref{p5-diff} into the second term of \eqref{p5-lemenergy1}, we get
	\begin{equation}\label{p5-lemenergy2}\begin{split}
\Re \left\{ \int_{\Gamma_1} \eta \overline{u}_t d\Gamma + \int_{\Gamma_2}\left(\xi\partial_{\nu_2} \overline{w}_t+\zeta\overline{w}_t\right) d \Gamma\right\}=\frac{1}{2} \frac{d}{dt}\left\{\int_{\Gamma_1}|\eta|^2 d\Gamma + \int_{\Gamma_2}\left(|\xi|^2+|\zeta|^2  \right)d\Gamma\right\} \\ + \int_{\Gamma_1} |\eta|^2 d\Gamma + \int_{\Gamma_2}\left(|\xi|^2+|\zeta|^2  \right)d\Gamma.\end{split}
\end{equation}
Combining equations \eqref{p5-lemenergy1} and \eqref{p5-lemenergy2}, we obtain 
$$
E^\prime (t)=- \int_{\Gamma_1} |\eta|^2d\Gamma-\int_{\Gamma_2} \left(|\xi|^2 + |\zeta|^2\right) d\Gamma.
$$
\end{proof}\\
Thus, from Lemma \ref{p5-lemenergy}, we deduce that the system \eqref{p5-sys2}-\eqref{p5-diff} is dissipative in the sense that the energy $E(t)$ is non-increasing with respect to time variable $t.$

\begin{rk}
	{\rm Note that the dissipative mechanism of the system \eqref{p5-sys2}-\eqref{p5-diff} remains constructed by a single dynamical boundary control of the wave or the plate.
}\hfill
\end{rk}

\subsection{Motivation and Aims} Dynamical boundary controls refer to approaches in real-life applications of mathematical physics and engineering. Since {\"O}. Morg{\"u}l proposed such a damping mechanism on the boundary of elastic beams 
(see, for instance, \cite{morgul1992dynamic}, 
 \cite{morgul1992dynamic2} and \cite{morgul1994control}), many authors have been interested in studying similar problems in the context of 
 plates (see \cite{article}, \cite{rao1998stabilization} and \cite{rao2005polynomial}). Recently in \cite{rao2015stability}, B. Rao et al. 
 considered the boundary stabilization 
 of a wave equation by means of singular dynamical boundary control. 
 In particular, they considered the following system on a given open bounded set $\Omega$ of $\R^N$ with boundary $\Gamma$ of class $C^2$ divided into two disjoint parts $\Gamma_0$ and $\Gamma_1:$ 
 \begin{equation}\label{p5-laila}
\left\{ \begin{array}{lll}
\displaystyle  u_{tt}-\Delta u =0, & \text{in} \ \  \Omega \times [0,T],\vspace{0.15cm}\\
\displaystyle u=0, & \text{on} \ \ \Gamma_0 \times [0,T],\vspace{0.15cm}\\
\displaystyle \frac{\partial u}{\partial {\nu} } + \eta =0, & \text{on} \ \ \Gamma_1 \times [0,T],\vspace{0.15cm}\\ 
\displaystyle \eta_t- u_t=- \eta, & \text{on} \ \ \Gamma_1 \times [0,T]
\end{array}
\right.
\end{equation}
with the initial condition:
\begin{equation} \label{laila-initial1}
 u(x,0)=u_0(x), \ \ u_t(x,0)=u_1(x), \ \ x \in \Omega, \ \ \eta(x,0)=\eta_0(x), \ \ x \in  \Gamma_1,
\end{equation}
where $\eta$ denotes the dynamical boundary control. 
The authors proved that the energy of the system \eqref{p5-laila}-\eqref{laila-initial1} does not decay uniformly (exponentially) to zero. However, using a multiplier method, they showed the polynomial stability of the system with an energy decay rate of $1/t$ under certain geometrical conditions for all smooth initial data. 
\\

 On the other hand, stimulated by many practical applications in the modeling and control of engineering, 
 transmission problems have attracted considerable attention over the past several years (see \cite{ammari2011study}, \cite{gong2017stabilization}, \cite{krstic2008backstepping}, \cite{wang2015stabilization}, \cite{wang2011stabilization}, \cite{hassine2018asymptotic}, \cite{hassine2016energy}, \cite{zhang2015stabilization} and the references therein). For instance, K. Ammari and S. Nicaise in \cite{ammari2010stabilization} considered the stabilization of a system coupling the wave equation with a Kirchhoff plate and damped through frictional boundary dissipation laws. Mainly, they deal with the following system:
 \begin{equation}\label{sergekais-sys}
\left\{ \begin{array}{llll}
\displaystyle  \partial_t^2 u_{1}(x,t)-\Delta u_1(x,t) =0, & \text{in} \ \  \Omega_1 \times (0,+\infty),\vspace{0.15cm}\\
\displaystyle \partial_t^2 u_{2}(x,t)+\Delta^2 u_2(x,t) =0, & \text{in} \ \   \Omega_2 \times (0,+\infty),\vspace{0.15cm}\\
\displaystyle u_i (x,0)=u_i^0(x), \ \ \partial_t u_i(x,0)=u_i^1(x), & \text{in} \ \   \Omega_i, \ \ i=1,2,\vspace{0.15cm}\\
\displaystyle u_1=u_2, \ \ \mathcal{B}_1 u_2=0, \ \ \mathcal{B}_2 u_2=\partial_{\nu_1} u_1, & \text{on} \ \ I \times (0,+\infty),\vspace{0.15cm}\\
\displaystyle \partial_{\nu_1} u_1=-\alpha_1 u_1 -  \partial_{t} u_1, & \text{on} \ \ \Gamma_1 \times  (0,+\infty),\vspace{0.15cm}\\ 
\displaystyle \mathcal{B}_1 u_2=-\beta \partial_{\nu_2} u_2- \partial_{\nu_2} \partial_{t} u_2, & \text{on} \ \ \Gamma_2 \times  (0,+\infty)\vspace{0.15cm},\\
\displaystyle \mathcal{B}_2 u_2=\alpha_2 u_2+ \partial_{t} u_2, & \text{on} \ \ \Gamma_2 \times  (0,+\infty)\vspace{0.15cm},
\end{array}
\right.
\end{equation}
where $\Omega$ refers to a bounded domain in $\R^2,$ as described in Subsection \ref{Systemwaveplate}. 
The unit normal vector of $\partial \Omega_i$ for $i=1,2$ and the unit tangent vector along $\partial \Omega_i$ are as described in Subsection \ref{Systemwaveplate}. The boundary operator $\mathcal{B}_j$ with $j=1,2,$ is also defined on $\partial \Omega_2$ as mentioned above.
Additionally, $\alpha_1, \ \alpha_2,$ and $\beta$ are three fixed positive constants.  
The authors proved, under certain geometric assumptions on the polygonal domains $\Omega_i, \ i=1,2,$ that the semi-group of the problem \eqref{sergekais-sys} is exponentially stable in the energy space when the boundary damping is of frictional type. 
\\

 To the best of our knowledge, the stabilization problem of a transmission system of wave and Kirchhoff plate equations with dynamical boundary controls remains unexplored in the literature and remains an open problem. It is essential to acknowledge that much of the motivation for studying the stabilization of this system arises from its significance in engineering and its strong physical foundation. This approach is expected to solve various control challenges in flexible structures composed of two physically distinct materials, where dynamical control plays a key role. Consequently, we focus on the transmission wave-plate model \eqref{p5-sys2}-\eqref{p5-diff} and the primary objective of this paper is to estimate the decay rate of the energy for this system under specific geometric conditions.\\

 Before going on, let us situate our work in comparison with previously published results in \cite{ammari2010stabilization}. Significantly, our work differs from \cite{ammari2010stabilization} in terms of the damping choices applied to the boundary. While \cite{ammari2010stabilization} utilized frictional boundary damping, we employ dynamical boundary controls, introducing novel mathematical challenges and leading to distinct results. One of the most intriguing issues in the mathematical theory of \eqref{p5-sys2}-\eqref{p5-diff} is the regularity of solutions, which we emphasize in this study. We have found that the proposed dynamical boundary conditions reduce the regularity of the solution compared to \cite{ammari2010stabilization}.  This reduction stems primarily from the lack of regularity in the boundary control $\partial_{\nu_1} u$ near the boundary $\Gamma_1$, posing significant challenges.  The mathematical theory of polygonal domains is known to be challenging, and it becomes even more complex when dealing with solutions of poor regularity. Additionally, the non-compactness of the resolvent of the associated operator introduces further difficulties in establishing the stability of \eqref{p5-sys2}-\eqref{p5-diff}. Furthermore, the uniform decay of the natural energy observed in \cite{ammari2010stabilization} does not hold in our case due to the implementation of dynamical boundary controls, leading to a different decay behavior.
 \\
 
 In part to answer these objections, we will first prove the well-posedness of the system \eqref{p5-sys2}-\eqref{p5-diff}. Then, we will impose sufficient geometric restrictions on the domains $\Omega_i, i=1,2$ (see Assumptions \ref{anglesomega1} and \ref{anglesomega2} below) to ensure the regularity of the solution to this system and guarantee the strict decay of the energy, i.e., $E(t)$
tends to zero as $t$ goes to infinity. The key question we are interested in, under Assumptions \ref{anglesomega1} and \ref{anglesomega2}, is the energy decay rate of the solution of \eqref{p5-sys2}-\eqref{p5-diff}. For this purpose, we will prove the lack of exponential decay of the system and establish a polynomial energy decay estimate of type $1/t$ for smooth initial data, provided that the exterior boundaries $\Gamma_1,$ $\Gamma_2$ and the interface $I$ satisfy an additional geometrical assumption described below (see Assumption \ref{geometriccon}). The frequency domain approach is the main tool used for the proof of polynomial stability. More precisely,
we combine a contradiction argument with a new multiplier technique to carry out a special analysis of the resolvent.

\subsection{Literature}

Dynamical controls are part of the indirect damping mechanisms proposed by Russell in \cite{russell1993general}, which do not arise from the insertion of damping terms into the original equations describing the mechanical motion, but by coupling those equations to further equations describing other processes in the structure. A significant number of papers dealing with the mathematical theory of dynamical boundary control in elastic structures have appeared. The emphasis is usually placed upon elastic beams modeled by the Euler-Bernoulli beam or Timoshenko beam, including \cite{morgul1992dynamic}, 
 \cite{morgul1992dynamic2} and \cite{morgul1994control} in our list of references. For example, {\"O}. Morg{\"u}l considered an Euler-Bernoulli beam, clamped to a rigid base at one end and free at the other end (see \cite{morgul1992dynamic}). To stabilize the beam vibrations, he proposed dynamic boundary control laws (i.e., dynamic actuators) at the free end of the beam 
 and proved that the beam vibrations decay exponentially to zero. A physical implementation of the dynamic control 
 may be used in pressurized gas tanks with servo-controlled actors 
 as well as in standard mass-spring dampers. Later, {\"O}. Morg{\"u}l studied the stabilization of the clamped-free Timoshenko beam with dynamic boundary control (see \cite{morgul1992dynamic2}). He proved that with the proposed control law, the beam vibrations decay uniformly and exponentially to zero. 
 Later on, {\"O}. Morg{\"u}l studied the motion of a flexible beam, which may model a flexible robot arm clamped to a rigid base at one end and free at the other end (see \cite{morgul1994control}). 
 To suppress the beam vibrations, he applied dynamic boundary control laws 
 to the free end of the beam 
  and proved that the beam vibrations decay asymptotically to zero under some assumptions on the actuator that generates this boundary control. 
  \\

 Let us start now by recalling some previous studies related to 
 Kirchhoff plates and wave equations with dynamical boundary controls, which are relevant to this study. 
 In 2005, A. Wehbe and B. Rao considered a plate equation in an open bounded domain $\Omega \subset \R^2$ with dynamical controls on the boundary $\Gamma$ with partition $\overline{\Gamma}=\overline{\Gamma_0} \cup \overline{\Gamma_1},$ where the dynamical controls are applied on $\Gamma_1$ (see \cite{rao2005polynomial}). They proved that the energy decays polynomially with the rate $1/t$ for all smooth initial data, assuming that the boundary $\Gamma$ satisfies the Multiplier Geometric Control Condition (MGC). This condition requires the existence of $\delta >0$ such that $m(x) \cdot \nu \geq \delta^{-1}$ for all $x \in \Gamma_1$ and $m(x) \cdot \nu \leq 0$ for all $x \in \Gamma_0,$ where $m(x)=x-x_0$ for $x_0$ fixed in $\R^2.$ In 2015, B. Rao et al. considered the stabilization of a wave equation by means of dynamical boundary controls (see \cite{rao2015stability}). They proved that a singular dynamical control applied to a portion of the boundary is sufficient to polynomially stabilize the wave equation with a decay rate $1/t$ under the MGC geometrical condition. Finally, in \cite{articleAkil}, M. Akil et al. studied the stabilization of a Kirchhoff plate equation with time delay added to the dynamical boundary controls. They found that the energy of the system decays polynomially with a decay rate of type $1/t$, provided that an appropriate condition on the delay terms is satisfied.
 \\

On the other hand, transmission systems are mathematical models that arise most naturally in the description of structures that are partly composed of interactive or interconnected materials. Transmission systems can be found in many practical applications, such as spacecraft \cite{biswas1989optimal}, satellite antennas \cite{luo1997shear}, road traffic \cite{lattanzio2011moving}, 
and many other interactive physical processes. 
Such problems have also received attention and there have been fruitful results concerning the control design and stability analysis of the solutions to different types of problems. In 2004, X. Zhang and E. Zuazua considered a  one-dimensional model for a heat-wave system arising from fluid-structure interaction and proved a sharp polynomial decay rate of type $1/t^2$ (see \cite{zhang2004polynomial}). Later, Q. Zhang et al. were concerned with the stabilization of coupled systems of Euler Bernoulli-beam or plate with a heat equation, where the heat equation plays the role of a controller of the whole system (see \cite{zhang2014stabilization}). 
 \\

 Here, we will review only the stabilization for transmission systems of wave and plate equations (or strings and beams), which have received a lot of attention in recent years. For example, K. Ammari and M. Mehrenberger conducted a detailed analysis of the resolvent of a string-beams network, leading to an exponential stability result for the system's energy in their study (see \cite{ammari2011study}).  For a transmission model arising in the control of noise, K. Ammari and S. Nicaise established the exponential stability of coupled wave-plate equations 
 under geometric conditions that lead to a flat interface between the two parts of the domain in which the wave and the plate equations evolve (see \cite{ammari2010stabilization}). 
One notable aspect of the research in \cite{ammari2010stabilization} is that the feedback controls are applied to both the wave and plate equations. Nevertheless, it was presented in \cite{li2018explicit} and \cite{guo2020energy} that different locations of internalized frictional damping bring out different kinds of energy decay rates as the dissipation acts through one equation of some transmission systems. Indeed, 
 Y.-F. Li et al. in \cite{li2018explicit} obtained an optimal polynomial decay rate of type $1/t$ when the frictional damping is actuated in the beam part, whereas an exponential decay of the energy is obtained when the frictional damping is effective in the string part. Furthermore, 
 Y.-P. Guo et al. in \cite{guo2020energy} established a polynomial energy decay of type $1/\sqrt{t}$ for a wave-frictionally damped plate system and an exponential energy decay for a frictionally damped wave-plate system. 

\subsection{Organization of the Paper}
This paper is organized as follows: In Section \ref{Wellposedtrans}, we formulate system \eqref{p5-sys2}-\eqref{p5-diff} into a first-order evolution equation and then deduce the well-posedness property by using semigroup theory. Some regularity results needed in the following sections are rigorously investigated in Section \ref{regularitytrans}. In Section \ref{StrongSta}, a general criteria of Arendt-Batty theorem is used to prove the strong stability of the system in the absence of compactness of the resolvent. We show the lack of exponential stability of \eqref{p5-sys2}-\eqref{p5-diff} in Section \ref{NonexponentialStabilityWave-Plateme}. Finally, Section \ref{polynomialstatrans} is devoted to study the energy decay rate under certain geometric conditions given by Theorem \ref{PolynomialStability}. We prove that the energy of our system has a polynomial decay rate of type $1/t.$\\

Let us finish this introduction with some notations used in the remainder of the paper: The $L^2(\Omega)-$norm will be denoted by $\norm{\cdot}_{L^2(\Omega)}.$ The usual norm and semi-norm of the Sobolev space $H^s(\Omega),$ $s\geq 0,$ are denoted by $\norm{\cdot}_{H^s(\Omega)}$ and $|\cdot|_{H^s(\Omega)},$ respectively. By $a\lesssim b,$ we mean that there exists a constant $\mathcal{C} >0$ independent of $a,$ $b$ and the natural parameter $n,$ such that $a \leq \mathcal{C} b.$  

\section{Well-posedness of the Problem}\label{Wellposedtrans}

This section is devoted to study the well-posedness property of system \eqref{p5-sys2}-\eqref{p5-diff} using the semigroup approach. We first introduce the following spaces:
$$H^1_{\ast}(\Omega_1)=\left\{u \in H^1(\Omega_1) \ | \ \  \int_{\Omega_1}u dx =0  \right\},$$
$$H^2_{\ast}(\Omega_2)=\left\{w \in H^2(\Omega_2) \ | \ \  \int_{\Omega_2}w dx=\int_{\Omega_2} \nabla w dx =0  \right\}$$
and the energy space $\HH$ by
$$
\HH=\left\{(u,v,\eta,w,z,\xi,\zeta) \in H^1_{\ast}(\Omega_1)\times L^2(\Omega_1)\times L^2(\Gamma_1) \times H^2_{\ast} (\Omega_2)\times L^2(\Omega_2) \times L^2(\Gamma_2) \times L^2(\Gamma_2) \ | \ \  u=w\  \ \text{on}  \ \ I \right\},
$$
which is endowed with the following usual inner product
\begin{equation}\label{p5-norm}
	\left(U,\tilde{U}\right)_{\HH}=\int_{\Omega_1}\left(\nabla u \cdot \nabla\overline{\tilde{u}}+v \overline{\tilde{v}}\right)dx + \int_{\Gamma_1} \eta \overline{\tilde{\eta}} d \Gamma+a(w,\overline{\tilde{w}})+\int_{\Omega_2} z\overline{\tilde{z}}dx + \int_{\Gamma_2} \left(\xi \overline{\tilde{\xi}}+ \zeta \overline{\tilde{\zeta}}\right) d \Gamma,
\end{equation}
where  $U=(u,v,\eta,w,z,\xi,\zeta) $, $\tilde{U} =(\tilde{u} ,\tilde{v} ,\tilde{\eta} ,\tilde{w} ,\tilde{z} ,\tilde{\xi} ,\tilde{\zeta})\in\HH$. We next define the linear unbounded  operator $\AA:D(\AA)\subset \HH\longmapsto \HH$  by
\begin{equation}\label{p5-domain}
	D(\AA)=\left\{\begin{array}{ll}\vspace{0.25cm}
		U=(u,v,\eta,w,z,\xi,\zeta)  \in \HH \ | \  \Delta u \in L^2(\Omega_1), \ \ v\in H^1_{\ast}(\Omega_1), \ \ \Delta^2 w \in L^2(\Omega_2), \ \ z \in  H^2_{\ast}(\Omega_2), \\\vspace{0.25cm} \partial_{\nu_1} u + \eta =0 \ \ \text{on} \ \ \Gamma_1, \ \  \mathcal{B}_1 w+ \xi =0 \ \ \text{on} \ \ \Gamma_2, \ \ \mathcal{B}_2w- \zeta =0 \ \ \text{on} \ \ \Gamma_2, \\\vspace{0.25cm}
v=z, \ \ \mathcal{B}_1w=0  \ \ \text{and} \ \ \mathcal{B}_2w=\partial_{\nu_1}u \ \ \text{on} \ \ I
	\end{array}\right\} 
\end{equation}and 
\begin{equation}\label{p5-opA}
	\AA U=\big(v, \Delta u, \gamma_1(v)-\eta, z, -\Delta^2 w, \gamma_{2,2}(z)-\xi, \gamma_{2,1}(z)-\zeta\big), \ \ \ \ \forall U=(u,v,\eta,w,z,\xi,\zeta)\in D(\AA),
\end{equation}
 where $\gamma_1:H^1(\Omega_1)\longmapsto L^2(\Gamma_1)$ and $\gamma_2:H^2(\Omega_2)\longmapsto L^2(\Gamma_2)\times L^2(\Gamma_2)$ are the trace operators such that $\gamma_1(\varphi)=\varphi|_{\Gamma_1}$ and $ \gamma_2(\psi)=\left(\gamma_{2,1}(\psi),\gamma_{2,2}(\psi)\right)=\left(\psi|_{\Gamma_2},\partial_{\nu_2} \psi|_{\Gamma_2}\right).$ 
\noindent Now, setting $U=(u, u_t ,\eta ,w, w_t , \xi, \zeta)$ as the state of system \eqref{p5-sys2}-\eqref{p5-diff}, we rewrite the problem into a first-order evolution equation 
\begin{equation}\label{p5-firstevo}
    \begin{cases}
    U_t =\AA U, \\
    U(0)=U_0,
    \end{cases}
\end{equation}
where  $U_0 =(u_0 ,u_1 ,\eta_0, w_0 ,w_1 ,\xi_0,\zeta_0) \in \HH$.
\begin{rk}The equation $$
\int_{\Omega_2} \nabla w \, dx = 0
$$ implies that $$
\int_{\Omega_2} (w_{x_1}, w_{x_2}) \, dx = 0,
$$ which further implies that both $$
\int_{\Omega_2} w_{x_1} \, dx = 0 \quad \text{and} \quad \int_{\Omega_2} w_{x_2} \, dx = 0.
$$ \end{rk} 
\begin{rk}\label{remarknormagg}
From \eqref{p5-1.3} and the fact that $2\Re \left(w_{x_1x_1}\overline{w}_{x_2x_2} \right)=|w_{x_1x_1}+w_{x_2x_2}|^2-|w_{x_1x_1}|^2-|w_{x_2x_2}|^2$, we remark that	\begin{equation*}\begin{array}{lll}
a(w,\overline{w})&=&\displaystyle \int_{\Omega_2} \left[|w_{x_1x_1}|^2+|w_{x_2x_2}|^2+\mu\left(w_{x_1x_1} \overline{w}_{x_2x_2}+w_{x_2 x_2} \overline{w}_{x_1 x_1}\right) +2(1-\mu)|w_{x_1 x_2}|^2\right]dx\vspace{0.25cm}\\
&=&\displaystyle \int_{\Omega_2} \left[|w_{x_1 x_1}|^2+|w_{x_2 x_2}|^2+2\mu \Re \left(w_{x_1 x_1} \overline{w}_{x_2 x_2}\right) +2(1-\mu)|w_{x_1 x_2}|^2\right] dx\vspace{0.25cm}\\
&=&\displaystyle \int_{\Omega_2} \left[(1-\mu)|w_{x_1 x_1}|^2+(1-\mu)|w_{x_2 x_2}|^2+\mu |w_{x_1 x_1}+w_{x_2 x_2}|^2 +2(1-\mu)|w_{x_1 x_2}|^2\right] dx\geq 0.
\end{array}\end{equation*}
\end{rk}

\begin{rk}
We present a brief description of the function spaces $H^1_{\ast}(\Omega_1)$ and $H^2_{\ast}(\Omega_2)$, which play a pivotal role in establishing a crucial result. Specifically, we demonstrate that if $\left(U, U \right)_{\mathcal{H}}=0$ of an element $U=\left(u,v,\eta,w,z,\xi,\zeta\right)$ in the function space $\mathcal{H}$ equals zero, then $U$ must be identically zero throughout the entire domain (i.e., $U=0$), where $\left\Vert U \right\Vert_\mathcal{H}^2$ serves as a norm on $\mathcal{H}$. To illustrate this, we consider $\left(U, U \right)_{\mathcal{H}}=0$. We then observe that the following equality holds:
\begin{equation*}
\int_{\Omega_1} \left(|\nabla u |^2+|v|^2 \right)dx + \int_{\Gamma_1} |\eta|^2 d \Gamma+a(w,\overline{w})+\int_{\Omega_2} |z|^2dx + \int_{\Gamma_2} \left(|\xi|^2 + |\zeta|^2 \right) d \Gamma=0.
\end{equation*}
From this, we deduce the following:
\begin{enumerate}
    \item 
$\int_{\Omega_1} |\nabla u|^2 dx +a(w,\overline{w})=0$, which implies that $\nabla u=0$ in $\Omega_1$ and $w_{x_1x_1}=w_{x_2x_2}=w_{x_1 x_2}=0$ in $\Omega_2$.\\

\item $\eta=0$ on $\Gamma_1$, $\xi=\zeta=0$ on $\Gamma_2$, $v=0$ in $\Omega_1$, and $z=0$ in $\Omega_2$. \end{enumerate}
To establish the final results, we make use of the following:
\begin{enumerate}
    \item From the condition $\int_{\Omega_1}u dx =0$, we deduce that $u=0$ in $\Omega_1$.\\

\item By applying Remark \ref{remarknormagg}, we infer that $w_{x_1x_1}=w_{x_2x_2}=w_{x_1x_2}=0$, leading to $w_{x_1}=c_1$ and $w_{x_2}=c_2$, where $c_1$ and $c_2$ are constants. Uzing the fact that $\int_{\Omega_2} \nabla w dx =0$, we further conclude that $w_{x_1}=w_{x_2}=0$. Moreover, $\int_{\Omega_2} w dx=0$ implies that $w=0$ in $\Omega_2$.
\end{enumerate}
This implies that both $u$ and $w$ are identically zero in their respective domains, i.e., $u=0$ in $\Omega_1$ and $w=0$ in $\Omega_2$. As a consequence, we arrive at the final result
\begin{equation*}
U=0.
\end{equation*}Thus, the spaces $H^1_{\ast}(\Omega_1)$ and $H^2_{\ast}(\Omega_2)$ play a crucial role in establishing this significant result.\end{rk}

\begin{rk}When replacing $H^1_{\ast}(\Omega_1)$ and $H^2_{\ast}(\Omega_2)$ with more general spaces, such as $H^1(\Omega_1)$ and $H^2(\Omega_2)$, Theorem \ref{PolynomialStability} may no longer hold. The original theorem's conclusion, derived from the equation $\int_{\Omega_1} |\nabla u|^2dx + a(w,\overline{w}) = 0$, where $(u,w) \in H^1_{\ast}(\Omega_1) \times H^2_{\ast}(\Omega_2)$, may not necessarily imply $u=0$ in $\Omega_1$ and $w=0$ in $\Omega_2$ when considering the broader function spaces. \end{rk}

For further purposes, set the Hilbert space $\mathsf{H}$ by 
$$\mathsf{H}=\big\{(f,g)\in H^1_{\ast}(\Omega_1)\times H^2_{\ast} (\Omega_2) \ | \ f=g \ \ \text{
on }\ \ I \big\},$$ equipped with the norm \begin{equation}\label{starsnorm}
	\left\Vert(f,g) \right\Vert_\mathsf{H}^2=\left\Vert \nabla f \right\Vert^2_{L^2(\Omega_1)}+a(g,\overline{g}).
\end{equation}
Note that using the compactness injection from $H^1(\Omega_1)$ into $L^2(\Omega_1)$ and from $H^2(\Omega_2)$ into $H^1(\Omega_2)$, we can easily show that $\left\Vert \nabla f \right\Vert^2_{L^2(\Omega_1)}$ is equivalent to the usual norm of $H^1(\Omega_1)$ on $H^1_{\ast}(\Omega_1)$ and $a(g,\overline{g})$ is equivalent to the usual norm of $H^2 (\Omega_2)$ on $H^2_{\ast} (\Omega_2).$


\begin{prop}\label{p5-aplusbmdissip}
		The unbounded linear operator $\AA$ is m-dissipative in the energy space $\HH$.
\end{prop}
\begin{proof}
	First, let $U=(u, v ,\eta ,w, z , \xi, \zeta)\in D(\AA)$. Using \eqref{p5-norm} and \eqref{p5-opA}, we have
\begin{equation*}\begin{split} \Re \left(\AA U,U\right)_\HH =\Re  \Bigg\{\int_{\Omega_1} \nabla v \cdot \nabla \overline{u}dx+ \int_{\Omega_1} \Delta u \overline{v}dx+ \int_{\Gamma_1}(v-\eta) \overline{\eta} d\Gamma+a(z,\overline{w})-\int_{\Omega_2} \Delta^2w \overline{z} dx\\+\int_{\Gamma_2}(\partial_{\nu_2} z-\xi) \overline{\xi}d\Gamma +\int_{\Gamma_2}(z-\zeta) \overline{\zeta}d\Gamma\Bigg\}.\end{split}\end{equation*}
We can proceed by integrating by parts and utilizing Green's formula \eqref{GF}. Given that $U \in D(\mathcal{A}),$ we have
\begin{equation}\label{p5-Adissipative}
	\Re \left( \AA U,U\right)_\HH =-\int_{\Gamma_1}|\eta|^2 d\Gamma - \int_{\Gamma_2}|\xi|^2 d\Gamma -\int_{\Gamma_2}|\zeta|^2 d\Gamma\leq 0,
	\end{equation}
		which implies that $\AA$ is dissipative. Now, let us prove that  $\AA$ is maximal. For this aim, let $F=(f^1 ,g^1,h^1,f^2,g^2,h^2,h^3)\\ \in \HH$, we look for $U=(u, v ,\eta ,w, z , \xi, \zeta) \in D(\AA)$ unique solution  of
		 \begin{equation}\label{p5-i-a=f}
		U-\AA U=F.
	\end{equation}
Equivalently, we have the following system
	\begin{eqnarray}
		u-v&=&f^1, \label{p5-f1} \ \ \text{in} \ \   \Omega_1 ,\\
		v-\Delta u&=&g^1,\label{p5-g1} \ \ \text{in} \ \   \Omega_1,\\
		\eta-v+\eta&=&h^1, \ \ \text{on} \ \   \Gamma_1, \label{p5-h1}\\
		w-z &=&f^2, \ \ \text{in} \ \   \Omega_2, \label{p5-f2}\\
		z+\Delta^2 w &=&g^2, \ \ \text{in} \ \   \Omega_2, \label{p5-g2}\\
		\xi-\partial_{\nu_2} z+\xi&=&h^2, \ \ \text{on} \ \   \Gamma_2, \label{p5-h2}\\
		\zeta-z+\zeta&=&h^3, \ \ \text{on} \ \   \Gamma_2, \label{p5-h3}
	\end{eqnarray}
with the following transmission and boundary conditions
\begin{equation}\label{p5-bc}
\left\{\begin{array}{lll}
\displaystyle u=w, \ \ \mathcal{B}_1 w=0, \ \ \mathcal{B}_2w=\partial_{\nu_1}u, & \text{on} \ \ I ,\vspace{0.15cm}\\ 
	\displaystyle \partial_{\nu_1} u + \eta =0, & \text{on} \ \ \Gamma_1 ,\vspace{0.15cm}\\ 
	\displaystyle \mathcal{B}_1 w+ \xi =0, & \text{on} \ \ \Gamma_2, \vspace{0.15cm}\\
	\displaystyle \mathcal{B}_2 w- \zeta =0, & \text{on} \ \ \Gamma_2.
\end{array}\right.
\end{equation}
It follows from \eqref{p5-f1}, \eqref{p5-h1}, \eqref{p5-f2}, \eqref{p5-h2} and \eqref{p5-h3} that 
\begin{equation}
	v=u-f^1 \ \ \text{in} \ \   \Omega_1 , \ \ 2 \eta= v+h^1 \ \ \text{on} \ \   \Gamma_1,  \ \ z=w-f^2 \ \ \text{in} \ \   \Omega_2, \ \ 2 \xi=\partial_{\nu_2} z+h^2  \ \ \text{on} \ \ \Gamma_2  \ \ \text{and} \ \ 2 \zeta=z+h^3  \ \ \text{on} \ \   \Gamma_2.
\end{equation}
By elimination of $v, \ z, \ \eta, \ \xi$ and $\zeta$ in \eqref{p5-g1}, \eqref{p5-g2} and \eqref{p5-bc}, we find that $u$ and $w$ satisfy the following system:
\begin{equation}\label{p5-strongdiss}
\left\{	\begin{array}{llll}
\displaystyle  u-\Delta u =f^1+g^1, & \text{in} \ \   \Omega_1,\vspace{0.15cm}\\
\displaystyle w+\Delta^2 w =f^2+g^2, & \text{in} \ \   \Omega_2,\vspace{0.15cm}\\
\displaystyle u=w, \ \ \mathcal{B}_1 w=0, \ \ \mathcal{B}_2w=\partial_{\nu_1}u, & \text{on} \ \ I,\vspace{0.15cm}\\
\displaystyle 2 \partial_{\nu_1} u + u =f^1-h^1, & \text{on} \ \ \Gamma_1,\vspace{0.15cm}\\ 
\displaystyle 2 \mathcal{B}_1 w+ \partial_{\nu_2} w =\partial_{\nu_2} f^2-h^2, & \text{on} \ \ \Gamma_2 \vspace{0.15cm},\\
\displaystyle 2 \mathcal{B}_2w- w=h^3-f^2, & \text{on} \ \ \Gamma_2.
\end{array}
\right.
\end{equation}
Let $\phi=(\varphi,\psi)\in \mathsf{H}.$ Multiplying the first equation of \eqref{p5-strongdiss} by $\overline{\varphi}$ and integrating over $\Omega_1$, multiplying the second equation of \eqref{p5-strongdiss} by  $\overline{\psi}$ and integrating over $\Omega_2$, then using Green's formula, we get 
\begin{equation}\label{p5-disseq1}
	\int_{\Omega_1} u \overline{\varphi}dx+	\int_{\Omega_1 } \nabla u \cdot \nabla \overline{\varphi}dx+ \frac{1}{2} \int_{\Gamma_1} u \overline{\varphi} d \Gamma-\int_I \partial_{\nu_1}u \overline{\varphi}d \Gamma=\int_{\Omega_1} F^1\overline{\varphi}dx+\frac{1}{2} \int_{\Gamma_1 }F^2 \overline{\varphi}d\Gamma,  
\end{equation}
\begin{equation}\label{p5-disseq2}
\begin{split}
	&\int_{\Omega_2} w \overline{\psi}dx+ a(w,\overline{\psi})+ \frac{1}{2} \int_{\Gamma_2} \left(\partial_{\nu_2} w \partial_{\nu_2} \overline{\psi} + w \overline{\psi} \right)d \Gamma-\int_I \left(\mathcal{B}_1 w \partial_{\nu_2} \overline{\psi} - \mathcal{B}_2 w \overline{\psi}\right)d \Gamma\\
	&\hspace{1cm}=\int_{\Omega_2} F^3\overline{\psi}dx+\frac{1}{2} \int_{\Gamma_2 }F^4 \partial_{\nu_2} \overline{\psi}d\Gamma+\frac{1}{2} \int_{\Gamma_2 }F^5  \overline{\psi}d\Gamma,  
\end{split}
\end{equation}
where $F^1=f^1+g^1, \ F^2=f^1-h^1, \ F^3=f^2+g^2, \  F^4=\partial_{\nu_2} f^2-h^2$ and $F^5=f^2-h^3.$
Adding the resulting equations and using the transmission conditions, we obtain a variational formulation of \eqref{p5-strongdiss}: Find $(u,w) \in \mathsf{H}$ such that
\begin{equation}\label{p5-vf11}
 b((u,w),(\varphi,\psi))=l(\varphi,\psi), \ \ \ \  \forall (\varphi,\psi)\in \mathsf{H},
\end{equation}
where
$$
b((u,w),(\varphi,\psi))=\int_{\Omega_1} u \overline{\varphi}dx+	\int_{\Omega_1 } \nabla u \cdot \nabla \overline{\varphi}dx+ \frac{1}{2} \int_{\Gamma_1} u \overline{\varphi} d \Gamma+\int_{\Omega_2} w \overline{\psi}dx+ a(w,\overline{\psi})+ \frac{1}{2} \int_{\Gamma_2} \left(\partial_{\nu_2} w \partial_{\nu_2} \overline{\psi} + w \overline{\psi} \right)d \Gamma
$$
and
$$
l(\varphi,\psi)= \int_{\Omega_1} F^1\overline{\varphi}dx+ \frac{1}{2} \int_{\Gamma_1 }F^2 \overline{\varphi}d\Gamma+ \int_{\Omega_2} F^3\overline{\psi}dx+\frac{1}{2} \int_{\Gamma_2 }F^4 \partial_{\nu_2} \overline{\psi}d\Gamma+\frac{1}{2} \int_{\Gamma_2 }F^5  \overline{\psi}d\Gamma.
$$
	It is easy to see that  $b$ is a sesquilinear, continuous, and coercive form on the space $\mathsf{H}\times \mathsf{H}$ and $l$ is an antilinear and continuous form on $\mathsf{H}$. Then, it follows by Lax-Milgram's theorem that \eqref{p5-vf11} admits a unique solution $(u,w)\in \mathsf{H} $. By choosing $\varphi \in C^{\infty}_0 (\Omega_1),$ $\psi=0$ in \eqref{p5-vf11} and applying Green's formula, we have 
	$$\int_{\Omega_1} \left( u-\Delta u \right) \overline{\varphi}dx= \int_{\Omega_1} \left( f^1+g^1 \right) \overline{\varphi}dx, \ \ \ \  \forall \varphi\in C^{\infty}_0 (\Omega_1),$$ which implies that the first equation of \eqref{p5-strongdiss} holds in the sense of distributions in $\Omega_1$ and hence is satisfied in $L^2(\Omega_1).$ As $u-f^1-g^1$ belongs to $L^2(\Omega_1),$ the same holds for $\Delta u,$ i.e., $\Delta u \in L^2(\Omega_1).$ In the same way, choosing $\varphi=0$ and $\psi \in C^{\infty}_0 (\Omega_2)$ in \eqref{p5-vf11}, we see that the second equation of \eqref{p5-strongdiss} holds as equality in $L^2(\Omega_2)$ and therefore $\Delta^2 w \in L^2(\Omega_2).$ Now, let us define the spaces 
	$$H^1_{\ast,I}(\Omega_1)=\left\{f \in H^1_\ast(\Omega_1) \ | \ \  f =0\  \ \text{on}  \ \ I \right\},$$ and 
$$H^2_{\ast,I}(\Omega_2)=\left\{f \in H^2_\ast(\Omega_2) \ | \ \  f=\partial_{\nu_2} f  =0\  \ \text{on}  \ \ I \right\}.$$ By taking $\varphi \in H^1_{\ast,I}(\Omega_1),$ $\psi=0$ and applying Green's formula in \eqref{p5-vf11}, we calculate
$$\int_{\Gamma_1} ( 2 \partial_{\nu_1} u + u - f^1+h^1 ) \overline{\varphi}d \Gamma=0, \ \ \ \  \forall \varphi\in H^{\frac{1}{2}}_{\ast,I}(\Gamma_1),$$ where $H^{\frac{1}{2}}_{\ast,I}(\Gamma_1)$ is the corresponding trace space of $H^1_{\ast,I}(\Omega_1)$ through the operator $\gamma_1.$ This implies that $\\2 \partial_{\nu_1} u + u - f^1+h^1 \in (H^{\frac{1}{2}}_{\ast,I}(\Gamma_1))^{\perp}$ in $L^2(\Gamma_1),$ and since $H^{\frac{1}{2}}_{\ast,I}(\Gamma_1)$ is dense in $L^2(\Gamma_1),$ we deduce that $u$ satisfies
$$2 \partial_{\nu_1} u + u =f^1-h^1, \ \ \ \  \text{on} \ \ \Gamma_1.$$ Similarly, by taking $\varphi=0,$ $\psi \in H^2_{\ast,I}(\Omega_2)$ and applying Green's formula in \eqref{p5-vf11}, we deduce that $w$ satisfies 
$$2 \mathcal{B}_1 w+ \partial_{\nu_2} w =\partial_{\nu_2} f^2-h^2, \ \ \ \  \text{on} \ \ \Gamma_2,$$
as well as 
	$$2 \mathcal{B}_2w- w=h^3-f^2, \ \ \ \  \text{on} \ \ \Gamma_2.$$ Coming back to \eqref{p5-vf11} and again applying Green's formula, we calculate
	$$\int_{I} \mathcal{B}_1 w \partial_{\nu_2}\overline{\psi}d \Gamma - \int_{I} \left( \mathcal{B}_2w - \partial_{\nu_1} u \right) \overline{\psi}d \Gamma=0, \ \ \ \  \forall \psi \in H^{\frac{3}{2}}_{\ast}(I),$$ where $H^{\frac{3}{2}}_{\ast}(I)$ is the corresponding trace space of $H^2_\ast(\Omega_2)$ through the operator $\psi \longmapsto \psi|_{I}.$ Due to the density of $H^{\frac{3}{2}}_{\ast}(I)$ into $L^2(I),$ we can easily check that $u$ and $w$ satisfy the transmission conditions of \eqref{p5-strongdiss}. Therefore, we deduce that system \eqref{p5-strongdiss} has a unique solution $(u,w) \in \mathsf{H}$ such that $\Delta u \in L^2(\Omega_1)$ and $\Delta^2 w \in L^2(\Omega_2).$ Finally, by setting 
	\begin{equation*}
v:=u-f^1, \ \eta:=-\partial_{\nu_1} u, \ z:=w-f^2, \ \xi:=-\mathcal{B}_1 w \ \ \text{and} \ \ \zeta:=\mathcal{B}_2 w,	
\end{equation*} we conclude that there exists a unique $U=(u,v,\eta,w,z,\xi,\zeta) \in D(\AA)$ solution of equation \eqref{p5-i-a=f} and thus the operator $\mathcal{A}$ is m-dissipative on ${\mathcal{H}}$. The proof is thus complete.
	\end{proof}\\\linebreak
According to Lumer-Philips theorem (see \cite{pazy2012semigroups}), Proposition \ref{p5-aplusbmdissip} implies that the operator $\AA$ generates a $C_{0}$-semigroup of contractions $\left(e^{t\AA}\right)_{t \geq 0}$ in $\HH$ which gives the well-posedness of \eqref{p5-firstevo}. Then, we have the following result:
\begin{Thm}
	For all $U_0 \in \HH$,  system \eqref{p5-firstevo} admits a unique weak solution $$U(t) \in C^0 (\R^+ ,\HH).
		$$ Moreover, if $U_0 \in D(\AA)$, then the system \eqref{p5-firstevo} admits a unique strong solution $$U(t) \in C^0 (\R^+ ,D(\AA))\cap C^1 (\R^+ ,\HH).$$
\end{Thm}

\begin{rk}
	 Note that Proposition \ref{p5-aplusbmdissip} remains true and system \eqref{p5-firstevo} admits a unique solution even if the domains $\Omega_1$ and $\Omega_2$  have no dynamical controls at their exterior boundaries $\Gamma_1$ and $\Gamma_2.$
\end{rk}

\section{Some Regularity Results}\label{regularitytrans}
 
 In the following sections, our goal is to establish energy estimates of the system \eqref{p5-firstevo}. Our stability results are based on a frequency domain approach combined with a multiplier technique and hence require some Green's formulas in $\Omega_1$ and $\Omega_2.$ In order to justify these formulas, we will need some regularity results proved in \cite{blum1980boundary} and \cite{dauge2006elliptic}. Hence, from now on, we assume that 
 $\Omega_1$ and $\Omega_2$ have a polygonal boundary in the sense that their boundary is piecewise smooth. For $i=1,2,$ we denote by $\omega_{i,j},$ $j=1, \dots, N_i$ the interior angles at the corners of $\Omega_i$ enclosed between two consecutive curves. These angles may vary within the range $0 < \omega_{i,j} \leq 2 \pi.$ To obtain the needed regularity, we require the following assumptions concerning the angles at the corners of $\Omega_1$ and $\Omega_2$:
\begin{assump}\label{anglesomega1}
{\rm The inner angles $\omega_{1,j} < \pi,$ for all $j=1, \dots, N_1.$  }
\end{assump}
\begin{assump}\label{anglesomega2}
{\rm There exists a minimal angle $\omega_0$ where $\omega_0$ depends on the Poisson coefficient $\mu$ (for instance, $\omega_0 \simeq 77.753311 \dots^{\circ}$ when $\mu=0.3$) such that $\omega_{2,j} < \omega_0,$ for all $j=1, \dots, N_2.$ }
\end{assump}

The regularity results are summarized by the following proposition:

\begin{prop}\label{p5-propregsol}
		Let $U=(u,v,\eta,w,z,\xi,\zeta) \in D(\AA).$ Assume that $\Omega_1$ and $\Omega_2$  are polygonal domains as described above, satisfying Assumptions \ref{anglesomega1} and \ref{anglesomega2}, respectively. Then, there exists $\varepsilon \in (0,\frac{1}{2})$ such that $u$ belongs to $H^{\frac{3}{2} - \varepsilon}(\Omega_1)$ and there exists a sequence $\left\{w_k\right\}_{k \geq 0}  \subset H^4(\Omega_2)$ such that $w_k$ converges to $w$ in $H^2(\Omega_2)$ and $\Delta^2 w_k$ converges to $\Delta^2 w$ in $L^2(\Omega_2).$
\end{prop}

Now, let us introduce some lemmas that are sufficient to prove our regularity result.

 \begin{lem}\label{p5-lemregu}
	Assume that $\Omega_1$  is a polygonal domain as described above, satisfying Assumption \ref{anglesomega1}. Then, there exists $\varepsilon \in (0,\frac{1}{2})$ such that the solution $y \in H^1(\Omega_1)$ of 
\begin{equation}\label{p5-y1}
\left\{\begin{array}{lll}
	\displaystyle \Delta y \in L^2(\Omega_1), \vspace{0.15cm}\\
	\displaystyle \partial_{\nu_1} y =v_1 \in L^2(\Gamma_1), & \text{on} \ \ \Gamma_1 ,\\
\displaystyle y =v_2 \in H^{\frac{3}{2}}(I), & \text{on} \ \ I ,
\end{array}\right.
\end{equation}
belongs to $H^{\frac{3}{2} - \varepsilon}(\Omega_1).$
\end{lem}
\begin{proof}
We define the operator: 
\begin{align*}
P: H^{1+s}(\Omega_1)&\rightarrow \Xi^s(\Omega_1):= H^{s-1}(\Omega_1) \times H^{s+\frac{1}{2}}(I) \times H^{s-\frac{1}{2}}(\Gamma_1)\\
y&\mapsto P(y):=(\Delta y, y |_{I},\partial_{\nu_1} y |_{\Gamma_1} )
\end{align*}
for all $0<s<1$ such that $s \neq \frac{1}{2}.$
From the elliptic regularity theory (see Theorem 23.3 of \cite{dauge2006elliptic}), we deduce that there exists $\varepsilon \in (0,\frac{1}{2})$ such that for any $s=\frac{1}{2}-\varepsilon,$ the operator $P$ is $[0,s]-$regular in the sense that if $y \in H^1(\Omega_1)$ is such that  $P(y) \in \Xi^s(\Omega_1),$  then $y \in H^{1+s}(\Omega_1).$ This implies that the solution  $y \in H^1(\Omega_1)$ of system \eqref{p5-y1} belongs to $H^{1+s}(\Omega_1),$ hence $y$ belongs to the space $H^{\frac{3}{2} - \varepsilon}(\Omega_1)$ for some $\varepsilon \in (0,\frac{1}{2}).$
\end{proof}

\begin{lem}\label{p5-lemregw}
 Assume that $\Omega_2$  is a polygonal domain as described above, satisfying Assumption \ref{anglesomega2}. Then, for any $\varepsilon \in (0,\frac{1}{2})$ and $y \in H^2(\Omega_2)$ solution of 
\begin{equation}\label{p5-y2}
\left\{\begin{array}{lll}
	\displaystyle \Delta^2 y \in L^2(\Omega_2),\vspace{0.15cm}\\
	\displaystyle \mathcal{B}_1 y =0, & \text{on} \ \ I, \vspace{0.15cm}\\
	\displaystyle \mathcal{B}_2 y =v_3 \in H^{-\varepsilon} (I), & \text{on} \ \ I,\vspace{0.15cm}\\
	\displaystyle \mathcal{B}_1 y =v_4 \in L^2(\Gamma_2), & \text{on} \ \ \Gamma_2,\vspace{0.15cm}\\
	\displaystyle \mathcal{B}_2 y =v_5 \in L^2(\Gamma_2), & \text{on} \ \ \Gamma_2,\vspace{0.15cm}
\end{array}\right.
\end{equation}
there exists a sequence $\left\{y_k\right\}_{k \geq 0}  \subset H^4(\Omega_2)$ such that $y_k$ converges to $y$ in $H^2(\Omega_2)$ and $\Delta^2 y_k$ converges to $\Delta^2 y$ in $L^2(\Omega_2).$
\end{lem}
\begin{proof}
Given $v_4,v_5 \in L^2(\Gamma_2)$ and $v_3 \in H^{-\varepsilon}(I)$, by density, there exist three sequences $ \left\{v_{4,k}\right\}_{k \geq 0},$ $ \left\{v_{5,k}\right\}_{k \geq 0}$ $\subset C_c^{\infty}(\Gamma_2)$ and  $ \left\{v_{3,k}\right\}_{k \geq 0}$ $\subset C_c^{\infty}(I)$ such that 
  \begin{equation}\label{p5-convergence}
v_{4,k} \underset{k \rightarrow \infty}{\longrightarrow} v_4, \ \ v_{5,k} \underset{k \rightarrow \infty}{\longrightarrow} v_5 \ \ \text{in} \ \ L^2(\Gamma_2) \ \ \ \text{and} \ \ \ v_{3,k} \underset{k \rightarrow \infty}{\longrightarrow} v_3 \ \ \text{in} \ \ H^{-\varepsilon }(I). 
\end{equation}
Then we consider the unique solution $y_k$ of \begin{equation}\label{p5-y22}
\left\{\begin{array}{lll}
	\displaystyle y_k + \Delta^2 y_k = y + \Delta^2 y \in L^2(\Omega_2), & \text{in} \ \ \Omega_2,\vspace{0.15cm}\\
	\displaystyle \mathcal{B}_1 y_k =0, & \text{on} \ \ I, \vspace{0.15cm}\\
	\displaystyle \mathcal{B}_2 y_k =v_{3,k}, & \text{on} \ \ I,\vspace{0.15cm}\\
	\displaystyle \mathcal{B}_1 y_k =v_{4,k}, & \text{on} \ \ \Gamma_2,\vspace{0.15cm}\\
	\displaystyle \mathcal{B}_2 y_k =v_{5,k}, & \text{on} \ \ \Gamma_2.\vspace{0.15cm}
\end{array}\right.
\end{equation} 
As the system $(\Delta^2, \mathcal{B}_1, \mathcal{B}_2)$ is a strongly elliptic system, we can use the elliptic regularity theory  (see Theorem 2 of \cite{blum1980boundary}). This allows us to deduce that $y_k \in H^4(\Omega_2)$ and that there exists a positive constant $C >0$ such that 
\begin{equation}\label{p5-normineqreg}
\norm{y-y_k}_{H^2(\Omega_2)} \leq C \left( \norm{v_4-v_{4,k}}_{L^2(\Gamma_2)} + \norm{v_5-v_{5,k}}_{L^2(\Gamma_2)} + \norm{v_3-v_{3,k}}_{H^{-\varepsilon }(I)} \right).
\end{equation}
By \eqref{p5-convergence} and \eqref{p5-normineqreg}, we deduce that  $y_k$ converges to $y$ in $H^2(\Omega_2).$ Thanks to \eqref{p5-y22}, we have $\Delta^2 y_k=y+\Delta^2 y -y_k,$ which converges to $\Delta^2 y$ in $L^2(\Omega_2).$
\end{proof}
\\

\noindent \textbf{Proof of Proposition \ref{p5-propregsol}.}
Indeed, $u$ may be seen as the unique solution $u \in H^1(\Omega_1)$ of 
\begin{equation*}
\left\{\begin{array}{lll}
	\displaystyle \Delta u \in L^2(\Omega_1),\vspace{0.15cm}\\
	\displaystyle \partial_{\nu_1} u =-\eta \in L^2(\Gamma_1), & \text{on} \ \ \Gamma_1 ,\\
\displaystyle u =w \in H^{\frac{3}{2}}(I), & \text{on} \ \ I ,
\end{array}\right.
\end{equation*}
and using Lemma \ref{p5-lemregu}, we conclude that $u \in H^{\frac{3}{2}- \varepsilon}(\Omega_1)$ for some $\varepsilon \in (0,\frac{1}{2}).$ 
By the definition of $D(\AA),$ we notice that $w$ is a solution of system \eqref{p5-y2} with $v_3=\partial_{\nu_1} u,$ $v_4=-\xi$ and $v_5=\zeta.$ We know that $v_3$ belongs to $H^{-\varepsilon}(I)$ for some $\varepsilon >0,$ while the $L^2(\Gamma_2)$ regularity of $v_4$ and $v_5$ follows from the regularity $\xi$ and $\zeta \in L^2(\Gamma_2).$ We then conclude using Lemma \ref{p5-lemregw}.
\begin{rk}
The situation is much more complicated for the regularity of $w$ due to the lack of regularity of the boundary conditions $\mathcal{B}_1 w =-\xi$ and $\mathcal{B}_2 w =\zeta$ on $\Gamma_2,$ as $\xi$ and $\zeta$ belong at most to $L^2(\Gamma_2),$ which is not sufficient to achieve the required regularity of $w.$ For this reason, we use arguments inspired by Lemma 3.1 of \cite{rao1993stabilization} and Lemma 3.2 of \cite{ammari2010stabilization}.
\end{rk}
\section{Strong Stability with Non-compact Resolvent}\label{StrongSta}

 In this section, we will prove the strong stability of the system \eqref{p5-sys2}-\eqref{p5-diff} in the sense that its energy $E(t)$ converges to zero as $t$ goes to infinity for all initial data in $\HH.$
As the resolvent of $\AA$ is not compact, classical methods such as Lasalle's invariance principle \cite{slemrod1989feedback} or the spectrum decomposition theory of Benchimol \cite{benchimol1978note} are not applicable in this case. Instead, we will prove the strong stability using a more general criteria of Arendt-Batty \cite{arendt1988tauberian} which states that in a reflexive Banach space, a $C_{0}$-semigroup of contractions $(e^{t\mathcal{A}})_{t \geq 0 }$ is strongly stable if $\mathcal{A}$ has no eigenvalues on the imaginary axis, and $\sigma(\mathcal{A}) \cap i\mathbb{R}$ is countable, where $\sigma(\mathcal{A})$ denotes the spectrum of  $\mathcal{A}.$ Our main result in this section is summarized by the following theorem:
\begin{Thm}\label{p5-strong}
Assume that $\Omega_1$ and $\Omega_2$  are polygonal domains as described above. Assume also that Assumptions \ref{anglesomega1} and \ref{anglesomega2} hold. Then, the semigroup of contractions $e^{t\AA}$ is strongly stable in the energy space $\HH$ in the sense that 
	$$\lim_{t\to \infty}\norm{e^{t\AA}U_0}_{\HH}=0, \,\,\,\, \forall U_0\in \HH.$$ 
\end{Thm}
\noindent In the proof of Theorem \ref{p5-strong}, we shall use the following lemmas:

\begin{lem}\label{p5-inj} Assuming that Assumptions \ref{anglesomega1} and \ref{anglesomega2} hold, the operator $\AA$ has no pure imaginary eigenvalues.
\end{lem}
\begin{proof} Let $\lambda \in \R$ and let $U=(u,v,\eta,w,z,\xi,\zeta) \in D(\AA)$ such that 
	\begin{equation}\label{p5-A+BU=ilaU}
		\AA U=i\la U.
	\end{equation}Equivalently, we have the following system of equations
\begin{eqnarray}
	v&=&i\la u,\label{p5-f1k}  \ \ \text{in} \ \   \Omega_1, \\
	\Delta u&=& i\la v, \label{p5-f2k} \ \ \text{in} \ \   \Omega_1, \\
	v- \eta &=&i\la \eta ,\label{p5-f3k} \ \ \text{on} \ \   \Gamma_1, \\
	z&=&i\la w,\label{p5-f4k} \ \ \text{in} \ \   \Omega_2, \\
	-\Delta^2 w &=& i\la z,\label{p5-f5k} \ \ \text{in} \ \   \Omega_2, \\
	\partial_{\nu_2} z- \xi&=&i\la \xi ,\label{p5-f6k} \ \ \text{on} \ \   \Gamma_2, \\
	z- \zeta&=&i\la \zeta ,\label{p5-f7k} \ \ \text{on} \ \   \Gamma_2.
	\end{eqnarray}
By using equations \eqref{p5-A+BU=ilaU} and \eqref{p5-Adissipative}, a direct computation leads to the  following:

	\begin{equation*}
	0=\Re\left\{i \la \left\lVert U  \right\rVert^2_{\HH}\right\}=\Re\{\left(\AA U,U\right)_{\HH}\}=-\int_{\Gamma_1}|\eta|^2 d\Gamma-\int_{\Gamma_2}|\xi|^2d\Gamma-\int_{\Gamma_2}|\zeta|^2 d\Gamma,
	\end{equation*}
	which implies that 
	\begin{equation}\label{p5-etaxizeta}
\eta=0 \ \ \text{on} \ \ \Gamma_1 \ \ \ \text{and}\ \ \ \xi=\zeta=0 \ \ \text{on} \ \ \Gamma_2.
\end{equation}
It follows from \eqref{p5-domain} that 
\begin{equation}\label{p5-injcon1}
\partial_{\nu_1} u=0 \ \ \text{on} \ \ \Gamma_1 \ \ \ \text{and}\ \ \ \mathcal{B}_1 w=\mathcal{B}_2w=0 \ \ \text{on} \ \ \Gamma_2,
\end{equation}
and from \eqref{p5-f3k}, \eqref{p5-f6k} and \eqref{p5-f7k} that 
\begin{equation}\label{p5-injcon2}
v=0 \ \ \text{on} \ \ \Gamma_1 \ \ \ \text{and}\ \ \ z=\partial_{\nu_2} z=0 \ \ \text{on} \ \ \Gamma_2.
\end{equation}
Now, we need to consider two distinct cases:\\

\textbf{Case 1.} If $\la=0$, then $v=0 \ \text{in} \ \Omega_1$ and $z=0 \ \text{in} \ \Omega_2,$ and we obtain the following system 
\begin{eqnarray}
	-\Delta u&=&0, \label{p5-la01k} \ \ \text{in} \ \   \Omega_1 ,\\
	\Delta^2 w&=&0, \label{p5-la02k} \ \ \text{in} \ \   \Omega_2 ,\\
	\partial_{\nu_1} u &=&0,\label{p5-la03k} \ \ \text{on} \ \   \Gamma_1 ,\\
	\mathcal{B}_1 w=\mathcal{B}_2w&=&0,\label{p5-la04k} \ \ \text{on} \ \   \Gamma_2 ,
	\end{eqnarray}
with the following transmission conditions
\begin{equation}\label{p5-injtrans}
u=w, \ \ \mathcal{B}_1 w=0, \ \ \mathcal{B}_2w=\partial_{\nu_1}u \ \ \text{on} \ \ I.
\end{equation}
Multiplying equations \eqref{p5-la01k} and \eqref{p5-la02k} by $\overline{u}$ and $\overline{w}$ respectively, integrating over $\Omega_1$ and $\Omega_2,$ then using Green's formula and equations \eqref{p5-la03k} and \eqref{p5-la04k}, we get 
\begin{equation}\label{p5-injla0eq1}
	\int_{\Omega_1} |\nabla u|^2dx-	\int_{I} \partial_{\nu_1}u \overline{u} d\Gamma=0,  
\end{equation}
\begin{equation}\label{p5-injla0eq2}
 a(w,\overline{w})- \int_{I} \left( \mathcal{B}_1 w \partial_{\nu_2} \overline{w} - \mathcal{B}_2 w \overline{w} \right) d \Gamma=0. 
\end{equation}
Adding the resulting equations and taking \eqref{p5-injtrans} into consideration, we obtain
\begin{equation}\label{p5-injla0res}
 \int_{\Omega_1} |\nabla u|^2dx+a(w,\overline{w})=0,
\end{equation}
which leads to 
\begin{equation}\label{p5-injla0ass1}
u=0 \ \ \text{in} \ \ \Omega_1 \ \ \ \text{and} \ \ \ w=0 \ \ \text{in} \ \ \Omega_2.
\end{equation}
Hence, we get $$U=0.$$

 \textbf{Case 2.} If $\la \neq0,$ then using equations \eqref{p5-f1k}, \eqref{p5-f4k} and \eqref{p5-injcon2}, we get 
\begin{equation}\label{p5-injcon3}
u=0 \ \ \text{on} \ \ \Gamma_1 \ \ \ \text{and}\ \ \ w=\partial_{\nu_2} w=0 \ \ \text{on} \ \ \Gamma_2.
\end{equation}
Inserting \eqref{p5-f1k} in \eqref{p5-f2k}, we obtain
\begin{equation}
\left\{	\begin{array}{lll}
	\displaystyle	\la^2 u+\Delta u =0, & \text{in} \ \ \Omega_1,\vspace{0.15cm}\\
	\displaystyle	u=\partial_{\nu_1} u=0, & \text{on} \ \ \Gamma_1.
	\end{array}\right.
\end{equation}
Thus, from the above system and by using a unique continuation theorem (see \cite{COCV_2012__18_3_712_0}) , we obtain
\begin{equation}\label{p5-u0}
	u=0 \ \ \text{in} \ \ \Omega_1.
\end{equation}
Now, from \eqref{p5-injcon1}, we get 
\begin{equation}\label{p5-expB}
	\mathcal{B}_1 w=\Delta w+(1-\mu)\mathcal{C}_1w=0 \ \ \text{on} \ \ \Gamma_2 \ \ \text{and} \ \  \mathcal{B}_2 w=\partial_{\nu_2}\Delta w+(1-\mu)\partial_{\tau_2 }\mathcal{C}_2w=0 \ \ \text{on} \ \ \Gamma_2.
\end{equation}
Using \eqref{p5-injcon3} and the fact that $\nabla w= \partial_{\tau_2}w \tau_2 +\partial_{\nu_2}w \nu_2 \ \text{on} \ \Gamma_2 $, we obtain
\begin{equation}\label{p5-derivw}
	\nabla w =(w_{x_1},w_{x_2})=(0,0) \ \ \text{on} \ \ \Gamma_2 \times \Gamma_2 \ \ \text{and consequently} \ \ w_{x_1}=w_{x_2}=0 \ \ \text{on} \ \ \Gamma_2.
\end{equation}
From equations \eqref{p5-JEL}, \eqref{p5-injcon3} and \eqref{p5-derivw}, we get
\begin{equation}
	\mathcal{C}_1w=\mathcal{C}_2w =0 \ \ \text{on} \ \ \Gamma_2,
\end{equation}
consequently, from \eqref{p5-expB}, we get
\begin{equation}
	\Delta w =\partial_{\nu_2} \Delta w =0 \ \ \text{on} \ \ \Gamma_2.
		\end{equation}
Inserting \eqref{p5-f4k} into \eqref{p5-f5k}, we obtain
\begin{equation}\label{p5-1.28}
\left\{\begin{array}{lll}
	\displaystyle \la^2 w-\Delta^2 w=0, & \text{in} \ \ \Omega_2,\vspace{0.15cm}\\
\displaystyle	w=\partial_{\nu_2}w=\Delta w=\partial_{\nu_2 }\Delta w=0, & \text{on} \ \ \Gamma_2.
\end{array}\right.
\end{equation}
Again, from the above system and by using a unique continuation theorem (see \cite{COCV_2012__18_3_712_0}), we obtain
\begin{equation}\label{p5-w0}
	w=0 \ \ \text{in} \ \ \Omega_2.
\end{equation}
Consequently, from equations  \eqref{p5-f1k}, \eqref{p5-f4k}, \eqref{p5-etaxizeta}, \eqref{p5-u0} and \eqref{p5-w0}, we get
$$
U=0.
$$
The proof is thus complete.
\
		\end{proof}
\begin{rk}
	 The preceding result is indeed true with some slight modifications, where only two dynamical controls $\xi$ and $\zeta$ are applied. First, one can show that $w=0$ in $\Omega_2.$ Then, due to some useful information provided by the interface, we obtain that $u=0$ in $\Omega_1.$ The transmission conditions play an essential role in this case. However, this may happen under certain geometric conditions in the case where only one dynamical control $\eta$ is applied. In such cases, we get $u=0$ in $\Omega_1$, but the transmission conditions introduced are not enough to conclude that $w=0$ in $\Omega_2$.

\end{rk}
\begin{lem}\label{p5-sur}
	 Assume that Assumptions \ref{anglesomega1} and \ref{anglesomega2} hold. Then, the operator $i \la \mathcal{I} -\AA$ is surjective for all real number $\la \in \R.$
\end{lem}

\begin{proof}
For any given $F=(f^1,g^1,h^1,f^2,g^2,h^2,h^3)\in \HH,$ we are looking for $U=(u,v,\eta,w,z,\xi,\zeta) \in D(\mathcal{A})$ that solves the following system:
\begin{equation}\label{p5-surilaI}
(i \lambda \mathcal{I} - \AA )U = F,
\end{equation}
which is equivalent to the following equations:
\begin{eqnarray}
	i \la u-v&=&f^1,\label{p5-surf1k} \ \ \text{in} \ \   \Omega_1, \\
	i \la v - \Delta u&=& g^1, \label{p5-surf2k} \ \ \text{in} \ \   \Omega_1, \\
	i \la \eta -v+ \eta &=&h^1,\label{p5-surf3k} \ \ \text{on} \ \   \Gamma_1, \\
	i \la w -z&=&f^2,\label{p5-surf4k} \ \ \text{in} \ \   \Omega_2,\\
	i \la z + \Delta^2 w &=& g^2,\label{p5-surf5k} \ \ \text{in} \ \   \Omega_2,\\
	i \la \xi -\partial_{\nu_2} z+\xi&=&h^2,\label{p5-surf6k} \ \ \text{on} \ \   \Gamma_2,\\
	i \la \zeta -z+ \zeta&=&h^3,\label{p5-surf7k} \ \ \text{on} \ \   \Gamma_2.
	\end{eqnarray}
By eliminating $v, \ z, \ \eta, \ \xi$ and $\zeta$ from the above equations and using \eqref{p5-domain}, we obtain the following system:
\begin{equation}\label{p5-strongsur}
\left\{	\begin{array}{llll}
\displaystyle  \la^2u+\Delta u =-(i \la f^1+g^1), & \text{in} \ \   \Omega_1,\vspace{0.15cm}\\
\displaystyle \la^2w-\Delta^2 w =-(i \la f^2+g^2), & \text{in} \ \  \Omega_2,\vspace{0.15cm}\\
\displaystyle u=w, \ \ \mathcal{B}_1 w=0, \ \ \mathcal{B}_2w=\partial_{\nu_1}u, & \text{on} \ \ I,\vspace{0.15cm}\\
\displaystyle \partial_{\nu_1} u + \dfrac{i \la}{i \la+1} u =\dfrac{1}{i \la +1} (f^1-h^1), & \text{on} \ \ \Gamma_1,\vspace{0.15cm}\\ 
\displaystyle \mathcal{B}_1 w+ \dfrac{i \la}{i \la+1}\partial_{\nu_2} w =\dfrac{1}{i \la+1}(\partial_{\nu_2} f^2-h^2), & \text{on} \ \ \Gamma_2, \vspace{0.15cm}\\
\displaystyle \mathcal{B}_2w- \dfrac{i \la}{i \la+1}w=\dfrac{1}{i \la+1}(h^3-f^2), & \text{on} \ \ \Gamma_2.
\end{array}
\right.
\end{equation}
Let $\phi=(\varphi,\psi)\in \mathsf{H}.$ Multiplying the first equation of \eqref{p5-strongsur} by $\overline{\varphi}$ and integrating over $\Omega_1$, 
 and multiplying the second equation of \eqref{p5-strongsur} by  $\overline{\psi}$ and integrating over $\Omega_2$, we use Green's formula to get the following equations:
\begin{equation}\label{surj1}
-\la^2\int_{\Omega_1} u \overline{\varphi}dx+	\int_{\Omega_1 } \nabla u \cdot \nabla \overline{\varphi}dx+ \frac{i\la}{i\la+1} \int_{\Gamma_1} u \overline{\varphi} d \Gamma-\int_I \partial_{\nu_1}u \overline{\varphi}d \Gamma=\int_{\Omega_1} F^1\overline{\varphi}dx+\frac{1}{i\la+1} \int_{\Gamma_1 }F^2 \overline{\varphi}d\Gamma,  
\end{equation}
and 
\begin{equation}\label{surj2}
\begin{split}
	&-\la^2\int_{\Omega_2} w \overline{\psi}dx+ a(w,\overline{\psi})+ \frac{i\la}{i\la+1} \int_{\Gamma_2} \left(\partial_{\nu_2} w \partial_{\nu_2} \overline{\psi} + w \overline{\psi} \right)d \Gamma-\int_I \left(\mathcal{B}_1 w \partial_{\nu_2} \overline{\psi} - \mathcal{B}_2 w \overline{\psi}\right)d \Gamma\\
	&\hspace{1cm}=\int_{\Omega_2} F^3\overline{\psi}dx+\frac{1}{i\la+1}\int_{\Gamma_2 }F^4 \partial_{\nu_2} \overline{\psi}d\Gamma+\frac{1}{i\la+1}\int_{\Gamma_2 }F^5 \overline{\psi}d\Gamma,  
\end{split}
\end{equation}
where 
$$F^1=i\la f^1+g^1, \quad F^2=f^1-h^1, \quad  F^3=i\la f^2+g^2, \quad F^4=\partial_{\nu_2} f^2-h^2 \quad \text{and} \quad F^5=f^2-h^3.$$
Adding equations \eqref{surj1} and \eqref{surj2} and using the transmission conditions, we obtain 
\begin{equation}\label{surj3}
\begin{split}
&-\la^2\left(\int_{\Omega_1} u \overline{\varphi}dx + \int_{\Omega_2} w \overline{\psi}dx\right)+	\int_{\Omega_1 } \nabla u \cdot \nabla \overline{\varphi}dx+  i\la \Lambda  \int_{\Gamma_1} u \overline{\varphi}d\Gamma  \\&+a(w,\overline{\psi})+i\la \Lambda \int_{\Gamma_2} \left(\partial_{\nu_2} w \partial_{\nu_2} \overline{\psi} + w \overline{\psi} \right)d \Gamma
\\&=\int_{\Omega_1} F^1\overline{\varphi}dx+\Lambda \int_{\Gamma_1 }F^2 \overline{\varphi}d\Gamma + \int_{\Omega_2} F^3\overline{\psi}dx+ \Lambda\int_{\Gamma_2 }F^4 \partial_{\nu_2} \overline{\psi}d\Gamma+\Lambda\int_{\Gamma_2 }F^5 \overline{\psi}d\Gamma,  
\end{split}
\end{equation}
where 
\begin{equation}
\Lambda= \frac{1-i\la}{\la^2+1}.
\end{equation}
Here we note that Lax-Milgram's theorem cannot be applied because coercivity is not available. Therefore, we use a compact perturbation argument. For that purpose, let us introduce the sesquilinear form
\begin{equation}\label{p5-vf}
\begin{split}
& b_\la ((u,w), (\varphi,\psi))= \int_{\Omega_1 } \nabla u \cdot \nabla \overline{\varphi}dx+i\la \Lambda \int_{\Gamma_1} u \overline{\varphi}d\Gamma  
+a(w,\overline{\psi}) \\
&\hspace{1cm} +i\la \Lambda \int_{\Gamma_2} \left(\partial_{\nu_2} w \partial_{\nu_2} \overline{\psi} + w \overline{\psi} \right)d \Gamma, \quad    \forall \, (\varphi,\psi)\in \mathsf{H}.
\end{split}
\end{equation}
This sesquilinear form $b_\lambda$ is continuous and coercive on $ \mathsf{H}\times  \mathsf{H}$. Then, by Lax-Milgram's theorem, the operator 
\begin{equation} 
B_{\lambda}: \mathsf{H} \rightarrow \mathsf{H}^{\prime}: \boldsymbol{(u,w)} \rightarrow B_{\lambda}\boldsymbol{(u,w)},\nonumber
\end{equation} 
with 
$
B_{\lambda}(\boldsymbol{u},\boldsymbol{w})((\boldsymbol{\varphi}, \boldsymbol{\psi})) := b_{\lambda}((\boldsymbol{u},\boldsymbol{w}),(\boldsymbol{\varphi},\boldsymbol{\psi})) \ \ \forall (\varphi,\psi) \in \mathsf{H},
$ is an isomorphism, where $\mathsf{H}^{\prime}$ is the dual space of $\mathsf{H}.$ Now, let us set
\begin{equation*} 
R_{\lambda}:  \mathsf{H} \rightarrow \mathsf{H}^{\prime}: (\boldsymbol{u},\boldsymbol{w}) \rightarrow R_{\lambda}(\boldsymbol{u},\boldsymbol{w}), 
 \end{equation*}
 with 
 \begin{equation*}  
 R_{\lambda}(\boldsymbol{u},\boldsymbol{w})((\boldsymbol{\varphi},\boldsymbol{\psi}))= -\la^2\left(\int_{\Omega_1} u \overline{\varphi}dx + \int_{\Omega_2} w \overline{\psi}dx\right), \quad    \forall \, (\varphi,\psi)\in \mathsf{H}.
 \end{equation*}
As $R_{\lambda}$ is a compact operator, we deduce that $B_{\lambda}+R_{\lambda}$ is a Fredholm operator of index zero from $\mathsf{H}$ to $\mathsf{H}^{\prime}.$
Now by setting
  \begin{equation*} 
L_{\lambda}(\varphi,\psi)=\int_{\Omega_1} F^1\overline{\varphi}dx+ \Lambda \int_{\Gamma_1 }F^2 \overline{\varphi}d\Gamma + \int_{\Omega_2} F^3\overline{\psi}dx+\Lambda\int_{\Gamma_2 }F^4 \partial_{\nu_2} \overline{\psi}d\Gamma+\Lambda\int_{\Gamma_2 }F^5 \overline{\psi}d\Gamma,
 \end{equation*} 
we notice that  \eqref{surj3} is equivalent to 
 \begin{equation}\label{surj4}
 (B_{\lambda}+R_{\lambda})(\boldsymbol{u},\boldsymbol{w})=L_{\lambda}\quad  \textrm{in} \,\, \mathsf{H}^{\prime}.
 \end{equation}\\ 
 Hence, problem \eqref{surj3} admits a unique solution $(\boldsymbol{u},\boldsymbol{w}) \in \mathsf{H}$ if and only if $B_{\lambda}+R_{\lambda}$ is invertible. Since $B_{\lambda}+R_{\lambda}$ is a Fredholm operator, it is enough to prove that $B_{\lambda}+R_{\lambda}$ is injective, i.e., \begin{equation} \ker(B_{\lambda}+R_{\lambda})=\{0\}.\nonumber
 \end{equation}
Let us now fix $(\boldsymbol{u},\boldsymbol{w}) \in\ker(B_{\lambda}+R_{\lambda})$, which means it satisfies
\begin{equation}\label{surj5}
\begin{split}
&-\la^2\left(\int_{\Omega_1} \boldsymbol{u} \overline{\boldsymbol{\varphi}}dx + \int_{\Omega_2} \boldsymbol{w} \overline{\boldsymbol{\psi}}dx\right)+	\int_{\Omega_1 } \nabla \boldsymbol{u} \cdot \nabla \overline{\boldsymbol{\varphi}}dx+  i\la\Lambda  \int_{\Gamma_1} \boldsymbol{u} \overline{\boldsymbol{\varphi}}d\Gamma  \\
&+a(\boldsymbol{w},\overline{\boldsymbol{\psi}})+ i\la\Lambda \int_{\Gamma_2} \left(\partial_{\nu_2} \boldsymbol{w} \partial_{\nu_2} \overline{\boldsymbol{\psi}} + \boldsymbol{w} \overline{\boldsymbol{\psi}} \right)d \Gamma=0, \quad \forall \, (\boldsymbol{\varphi},\boldsymbol{\psi})
\in  \mathsf{H}.
\end{split}
\end{equation}
Thus, if we set 
$$\boldsymbol{v}=i\la \boldsymbol{u}, \quad \boldsymbol{\eta}=i\la \Lambda  \boldsymbol{u}, \quad \boldsymbol{z}=i\la \boldsymbol{w}, \quad 
\boldsymbol{\xi}=i\la \Lambda \partial_{\nu_2}\boldsymbol{w} \quad \text{and} \quad \boldsymbol{\zeta}=i\la \Lambda \boldsymbol{w},$$
 we can conclude that $ \boldsymbol{U}=(\boldsymbol{u},\boldsymbol{v},\boldsymbol{\eta},\boldsymbol{w},\boldsymbol{z},\boldsymbol{\xi},\boldsymbol{\zeta})\in D(\mathcal{A})$ is a solution of 
\begin{equation}(i\lambda-\mathcal{A})\boldsymbol{U}=0.\nonumber
 \end{equation} 
 Using Lemma \ref{p5-inj}, we deduce that $\boldsymbol{U}=0.$ This shows that $B_{\lambda}+R_{\lambda}$ is invertible, and therefore a unique solution $(u,w) \in \mathsf{H}$  of $(\ref{surj4})$ exists. At this stage, by setting 
 $$v=i\la u-f^1, \quad \eta=i\la \Lambda u + \Lambda (h^1 - f^1), \quad z=i\la w-f^2,$$
$$\xi=i\la \Lambda \partial_{\nu_2} w + \Lambda (h^2 - \partial_{\nu_2} f^2) \quad \text{and} \quad \zeta= i\la \Lambda w + \Lambda (h^3 - f^2),$$
 we conclude that $(u,v,\eta,w,z,\xi,\zeta)\in D(\mathcal{A})$ is a solution of \eqref{p5-surilaI} and the proof is thus complete.
\end{proof}
\\

\noindent \textbf{Proof of Theorem \ref{p5-strong}.}
Using Lemma \ref{p5-inj}, we conclude that the operator $\mathcal{A}$ has no pure imaginary eigenvalues. Additionally, by Lemma \ref{p5-sur}, we have $ Im (i\lambda \mathcal{I}- \AA)=\HH$ for all real numbers $\lambda \in \mathbb{R}.$  Therefore, the closed graph theorem of Banach implies that $\sigma(\mathcal{A}) \cap i\mathbb{R} = \emptyset.$ Following Arendt-Batty (see \cite{arendt1988tauberian}), the $C_{0}$-semigroup of contractions $(e^{t\mathcal{A}})_{t \geq 0 }$ is strongly stable and the proof is complete.

\section{Lack of Exponential Stability}\label{NonexponentialStabilityWave-Plateme}
In this section, we will prove that the system \eqref{p5-sys2}-\eqref{p5-diff} is not exponentially stable. Concerning the characterization of exponential stability of a $C_0-$semigroup of contractions, we rely on the following result due to Huang \cite{huang1985characteristic} and 
 Prüss \cite{10.2307/1999112}:
\begin{Thm}\label{Caractwaveplate}
		Let $A:\ D(A)\subset H\longrightarrow H$ generate a $C_0-$semigroup of contractions $\left(e^{tA}\right)_{t\geq0}$ on $H$. Then, the $C_0-$semigroup $\left(e^{tA}\right)_{t\geq0}$ is exponentially stable if and only if $i\R \subset\rho(A)$ and
		$$
		\limsup_{\la\in \R,\ \abs{\la}\rightarrow \infty}\|(i\la \mathcal{I}-A)^{-1}\|_{\mathcal{L}(H)}<\infty.
		$$
\end{Thm} The main result of this section is the following theorem:

\begin{Thm}\label{nonexsthwaveplate}
The $C_0-$semigroup  of contractions $(e^{t\AA})_{t \geq 0}$ is not uniformly stable in the energy space $\HH$.
\end{Thm}

\noindent According to Theorem \ref{Caractwaveplate} due to Huang \cite{huang1985characteristic} and 
 Prüss \cite{10.2307/1999112}, it is sufficient to prove that the resolvent of the operator $\AA$ is not uniformly bounded on the imaginary axis. 
For this aim, let us start with the following technical lemma:

\begin{lem}\label{nonex1wp}
Define the linear unbounded operator $\mathcal{O}_{\Delta,R} :D\left(\mathcal{O}_{\Delta,R} \right) \longmapsto L^2(\Omega_1)\times L^2(\Omega_2)$ by

\begin{equation}\label{defDTwp}
\begin{split}
D\left(\mathcal{O}_{\Delta,R} \right)=\left\{\begin{array}{ll}\vspace{0.25cm}(f,g) \in \mathsf{H} \ | \  \left(\Delta f ,\Delta^2 g\right)  \in L^{2}(\Omega_1)\times L^{2}(\Omega_2),   \ \ \partial_{\nu_1} f +f =0 \ \ \text{on} \ \ \Gamma_1, \\\vspace{0.25cm}  \mathcal{B}_1 g+ \partial_{\nu_2} g =0 \ \ \text{on} \ \ \Gamma_2,  \ \ \mathcal{B}_2 g- g =0 \ \ \text{on} \ \ \Gamma_2,  \ \  \mathcal{B}_1 g=0  \ \ \text{and} \ \ \mathcal{B}_2 g=\partial_{\nu_1} f \ \ \text{on} \ \ I \end{array}\right\},
\end{split}
\end{equation}
and

\begin{equation}
\label{defimageTwp}
\mathcal{O}_{\Delta,R} (f,g) =\left(-\Delta f, \Delta^2 g\right), \ \ \forall \, (f,g) \in D\left(\mathcal{O}_{\Delta,R} \right).
\end{equation}
Then, $\mathcal{O}_{\Delta, R} $ is a positive self-adjoint operator with a compact resolvent.
\end{lem}

\begin{proof}
To prove that $\mathcal{O}_{\Delta, R} $ is a positive self-adjoint operator, we will show that it is a symmetric m-accretive operator. For this purpose, we will divide the proof into steps:\\\\
\textbf{Step 1.} ($\mathcal{O}_{\Delta, R}$  is symmetric.) Indeed, 
for all $(f,g), \ (h,k) \in D\left(\mathcal{O}_{\Delta,R} \right)$, we have
\begin{equation}
\begin{split}
&\left(\mathcal{O}_{\Delta,R} (f,g), (h,k)\right)_{L^{2}(\Omega_1) \times L^{2}(\Omega_2)}=- \int_{\Omega_1}(\Delta f) \overline{h} dx +  \int_{\Omega_2}(\Delta^2 g) \overline{k} dx\\\\
&\hspace{3cm} = \int_{\Omega_1} \nabla f \cdot \nabla \overline{h}dx + \int_{\Gamma_1} f \overline{h} d \Gamma + a(g,\overline{k})+  \int_{\Gamma_2} \left(\partial_{\nu_2} g \partial_{\nu_2} \overline{k} + g \overline{k} \right)d \Gamma
\\\\
&\hspace{3cm}  = \left((f,g), \mathcal{O}_{\Delta,R} (h,k)\right)_{L^{2}(\Omega_1) \times L^{2}(\Omega_2)}.
\end{split}
\end{equation}
Thus, $\mathcal{O}_{\Delta,R}$ is  symmetric.\\\\
\textbf{Step 2.} ($\mathcal{O}_{\Delta,R}$ is $m-$accretive.) Indeed, for all $(f,g) \in D(\mathcal{O}_{\Delta,R} )$, we have
\begin{equation}
\begin{split}
 & \Re\left(\mathcal{O}_{\Delta,R} (f,g), (f,g)\right)_{L^{2}(\Omega_1) \times L^{2}(\Omega_2)}=\Re \left\{- \int_{\Omega_1}(\Delta f) \overline{f} dx + \int_{\Omega_2}(\Delta^2 g) \overline{g} dx\right\}\\\\
&\hspace{3cm} = \int_{\Omega_1} |\nabla f |^2dx + \int_{\Gamma_1} |f|^2 d \Gamma + a(g,\overline{g})+  \int_{\Gamma_2} \left(|\partial_{\nu_2} g|^2  + |g|^2  \right)d \Gamma \geq 0.
\end{split}
\end{equation}
Thus, $\mathcal{O}_{\Delta,R}$ is an accretive operator. Now, let $(F,G) \in L^{2} (\Omega_1)\times L^{2}(\Omega_2)$ and $\lambda >0$, looking for $(f,g) \in D\left(\mathcal{O}_{\Delta,R} \right)$ solution of 

\begin{equation}
\left(\lambda \mathcal{I} + \mathcal{O}_{\Delta,R} \right)(f,g) =(F,G).
\end{equation}
Equivalently, we have the following system:
\begin{eqnarray}
 \lambda f -  \Delta f = F,\label{a_1}\\
 \lambda g+  \Delta^2 g = G.\label{a_2}
\end{eqnarray}
Taking $(\varphi, \psi) \in \mathsf{H} $, then integrating after multiplying \eqref{a_1} by $\overline{\varphi}$ and \eqref{a_2} by $\overline{\psi}$ yields the two equations added as follows
\begin{equation}\begin{split}
 \int_{\Omega_1} \nabla f \cdot \nabla \overline{\varphi} dx + a(g,\overline{\psi}) +  \int_{\Gamma_1}  f   \overline{\varphi}d\Gamma  + \int_{\Gamma_2} \left(\partial_{\nu_2} g \partial_{\nu_2} \overline{\psi} + g \overline{\psi} \right)d \Gamma+ \lambda \int_{\Omega_1} f  \overline{\varphi}dx + \lambda \int_{\Omega_2} g   \overline{\psi}dx \\=  \int_{\Omega_1} F  \overline{\varphi}dx +  \int_{\Omega_2} G   \overline{\psi}dx.\end{split}
\end{equation} 
Let 
\begin{equation}\begin{split}
S \left((f,g) , (\varphi, \psi)\right)=\int_{\Omega_1} \nabla f \cdot \nabla \overline{\varphi} dx + a(g,\overline{\psi}) +  \int_{\Gamma_1}  f   \overline{\varphi}d\Gamma  + \int_{\Gamma_2} \left(\partial_{\nu_2} g \partial_{\nu_2} \overline{\psi} + g \overline{\psi} \right)d \Gamma\\+ \lambda \int_{\Omega_1} f  \overline{\varphi}dx + \lambda \int_{\Omega_2} g   \overline{\psi}dx,\end{split}
\end{equation} 
and 
\begin{equation}
L(\varphi, \psi)=  \int_{\Omega_1} F  \overline{\varphi}dx +  \int_{\Omega_2} G   \overline{\psi}dx.
\end{equation} 
It is easy to see that  $S$ is a sesquilinear, continuous, and coercive form on the space $\mathsf{H}\times \mathsf{H}$, and $L$ is an antilinear and continuous form on $\mathsf{H}$. Then, it follows by Lax-Milgram's theorem that  $S((f,g) , (\varphi, \psi)) = L(\varphi, \psi)$ admits a unique solution $(f,g)\in \mathsf{H}$. By classical elliptic regularity, we deduce that $(f,g)\in D(\mathcal{O}_{\Delta,R} )$ solution of system \eqref{a_1}-\eqref{a_2}. Thus, $\mathcal{O}_{\Delta,R}$ is m-accretive.\\\\
\textbf{Step 3.} ($\mathcal{O}_{\Delta,R}$ has a compact resolvent.) Let
$$
R_{\lambda}(\mathcal{O}_{\Delta,R})=\left(\lambda \mathcal{I} + \mathcal{O}_{\Delta,R}\right)^{-1}.
$$
Due to Sobolev embeddings, $R_{0}(\mathcal{O}_{\Delta,R})$ is compact. Then, using the following resolvent identity
$$
R_{\lambda}- R_{\mu}=(\mu -\lambda)R_{\mu}R_{\lambda},
$$
we deduce that the resolvent $(\lambda \mathcal{I} + \mathcal{O}_{\Delta,R})^{-1}$ of the operator $\mathcal{O}_{\Delta,R}$ is compact, and the proof is thus complete.\\
\end{proof}
\\
\noindent {\textbf {Proof of Theorem \ref{nonexsthwaveplate}.}}   
According to Theorem \ref{Caractwaveplate} due to Huang \cite{huang1985characteristic} and 
 Prüss \cite{10.2307/1999112}, it is sufficient to show that the resolvent of $\mathcal{A}$ is not uniformly bounded on the imaginary axis. In other words, it is enough to show the existence of a positive real number $M$ and some sequences 
$\lambda_n \in i\R$, $U_n=(u_n,v_n,\eta_n,w_n,z_n,\xi_n,\zeta_n)^{\top}\in D(\mathcal{A})$, and $F_n=(f_{n}^1,g_{n}^1,h_n^1,f_{n}^2,g_{n}^2, h_{n}^2,h_n^3)^{\top}\in \mathcal{H}$, where $n\in \N$, such that

\begin{equation}\label{p5-4.1l}
	(\lambda_n \mathcal{I}-\mathcal{A})U_n=F_n, \quad \forall n\in \N,
	\end{equation}
	
\begin{equation}\label{p5-4.2l}
\|U_n\|_{\mathcal{H}}\geq M, \quad \forall n\in \N,
\end{equation}

\begin{equation}\label{p5-4.3l}
\lim_{n\to \infty}\|F_n\|_{\mathcal{H}}=0.
\end{equation}
From Lemma \ref{nonex1wp}, we can consider 
the sequence of  eigenfunctions $(\varphi_n,\psi_n)_{n\in \N}$ (that form an orthonormal basis of $L^2(\Omega_1)\times L^2(\Omega_2)$) of the operator $\mathcal{O}_{\Delta,R}$, corresponding to the eigenvalues $(\mu_n^2)_{n\in \N}$, such that $\mu_n^2$ tends to infinity as $n$ goes to  infinity. Consequently, for all $n\in \N$, they satisfy the following system 

 \begin{equation}\label{SS123}
\left\lbrace\begin{array}{ll}
\displaystyle -\Delta \varphi_n=\mu_n ^2 \varphi_n, \hspace{1cm} &\hbox{in}\quad \Omega_1,\\[0.1in]
\displaystyle \Delta^2 \psi_n=\mu_n ^2 \psi_n,\hspace{1.02cm} &\hbox{in}\quad \Omega_2, \\[0.1in]
\varphi_n-\psi_n =0, \ \ \mathcal{B}_1 \psi_n=0, \ \ \mathcal{B}_2 \psi_n=\partial_{\nu_1} \varphi_n &  \text{on} \quad I ,\\[0.1in]
\partial_{\nu_1} \varphi_n +\varphi_n=0, & \text{on} \quad \Gamma_1,\\ [0.1in]
\mathcal{B}_1 \psi_n+ \partial_{\nu_2} \psi_n =0, &  \text{on} \quad \Gamma_2,\\[0.1in]
\mathcal{B}_2 \psi_n - \psi_n =0, \hskip1.2cm \qquad &\hbox{on}\quad \Gamma_2,
\end{array}
\right.
\end{equation}
with

\begin{equation}
\norm{(\varphi_n,\psi_n)}_{L^2(\Omega_1)\times L^2(\Omega_2)}^2= \int_{\Omega_1} | \varphi_n|^2 dx + \int_{\Omega_2} | \psi_n|^2 dx =1.
\label{normwp}
\end{equation}
Now, let us choose 

\begin{equation}\label{choices}
\displaystyle  u_n =\frac{\varphi_n}{i\mu_n}, \ \ v_n =\varphi_n, \ \ \eta_n =\frac{1}{i\mu_n} \gamma_1 (\varphi_n), \ \ w_n =\frac{\psi_n}{i\mu_n}, \ \ z_n =\psi_n, \ \ \xi_n =\frac{1}{i\mu_n} \gamma_{2,2}(\psi_n), \ \ \zeta_n= \frac{1}{i\mu_n} \gamma_{2,1}(\psi_n).
\end{equation}
So, by setting \begin{equation} F_n=\left(0,0,\frac{1}{i\mu_n} \gamma_1 (\varphi_n),0,0,\frac{1}{i\mu_n} \gamma_{2,2}(\psi_n),\frac{1}{i\mu_n} \gamma_{2,1}(\psi_n)\right),\end{equation} we deduce that 

\begin{equation}
U_n=\left(u_n, v_n,\eta_n, w_n, z_n, \xi_n, \zeta_n\right),
\end{equation} 
is the solution in $D(\mathcal{A})$ of the following equation 
 
\begin{equation}
(i\mu_n \mathcal{I}-\mathcal{A}) U_n=F_n. 
\end{equation}
Now, multiplying equation $\eqref{SS123}_1$  and $\eqref{SS123}_2$ by $\overline{\varphi_n}$ and $\overline{\psi_n}$ respectively, integrating by parts, we get

\begin{equation}\begin{split}
\displaystyle  \int_{\Omega_1} |\nabla \varphi_n|^2 dx +  \int_{\Gamma_1} |\varphi_n|^2 d\Gamma +  a(\psi_n,\overline{\psi_n})+  \int_{\Gamma_2} \left(|\partial_{\nu_2} \psi_n|^2  + |\psi_n|^2  \right)d \Gamma  \\ =
\mu_n^{2}\int_{\Omega_1}| \varphi_n|^2 dx+\mu_n^{2} \int_{\Omega_2}| \psi_n|^2 dx=\mu_n^{2}.\end{split}
\end{equation}
Since the norm defined on the left-hand side of the above equation is equivalent to the usual norm of $H^1(\Omega_1) \times H^2(\Omega_2)$ on $H^1(\Omega_1) \times H^2(\Omega_2),$ we get
\begin{equation}\label{normH12}
\displaystyle \norm{\left(\varphi_n,\psi_n\right)}_{H^1(\Omega_1) \times H^2(\Omega_2)}^2 \lesssim \mu_n^2.
\end{equation}
On the other hand, we have  
\begin{equation}
\begin{split}
\|U_n\|_{\mathcal{H}}^2&=\int_{\Omega_1}| \varphi_n|^2 dx+\int_{\Omega_2}| \psi_n|^2 dx \\
&\quad + \mu_n^{-2}  \int_{\Omega_1} |\nabla \varphi_n|^2 dx+ \mu_n^{-2} \int_{\Gamma_1} |\varphi_n|^2 d \Gamma +  \mu_n^{-2} a(\psi_n,\overline{\psi_n})+  \mu_n^{-2} \int_{\Gamma_2} \left(|\partial_{\nu_2} \psi_n|^2  + |\psi_n|^2  \right)d \Gamma  \geq 1,
\label{normboundedwp}
\end{split}
\end{equation} which implies that \eqref{p5-4.2l} holds for $M=1.$ By using the trace theorem of interpolation type (see Theorem 1.4.4 in \cite{liu1999semigroups} and Theorem 1.5.1.10 in \cite{grisvard1985elliptic}), we obtain

\begin{equation}\begin{split}\label{fnonexpwp}
\norm{F_n}^2_\mathcal{H}&=\displaystyle \mu_n^{-2} \norm{\varphi_n}^2_{L^2(\Gamma_1)}+\mu_n^{-2} \norm{\partial_{\nu_2} \psi_n}^2_{L^2(\Gamma_2)}+\mu_n^{-2} \norm{\psi_n}^2_{L^2(\Gamma_2)}
 \\&\lesssim\mu_n^{-2} \norm{\varphi_n}_{H^1(\Omega_1)}\norm{\varphi_n}_{L^2(\Omega_1)}+\mu_n^{-2} \norm{\psi_n}_{H^2(\Omega_2)}\norm{\psi_n}_{H^1(\Omega_2)}+\mu_n^{-2} \norm{\psi_n}_{H^1(\Omega_2)}\norm{\psi_n}_{L^2(\Omega_2)}.\end{split}
\end{equation}
Further, by using Theorem 4.17 in \cite{adams1975sobolev}, equation \eqref{normwp} and equation  \eqref{normH12}, we obtain from the above inequality that
\begin{equation}\begin{split}
\norm{F_n}^2_\HH &\lesssim \displaystyle \mu_n^{-2} \norm{\varphi_n}_{H^1(\Omega_1)}\norm{\varphi_n}_{L^2(\Omega_1)}+\mu_n^{-2} \norm{\psi_n}_{H^2(\Omega_2)}^{\frac{3}{2}}\norm{\psi_n}_{L^2(\Omega_2)}^{\frac{1}{2}}+\mu_n^{-2} \norm{\psi_n}_{H^2(\Omega_2)}^{\frac{1}{2}}\norm{\psi_n}_{L^2(\Omega_2)}^{\frac{3}{2}} \\  &\lesssim \mu_n^{-1} + \mu_n^{-\frac{1}{2}}+ \mu_n^{-\frac{3}{2}} \to 0 
 \quad \text{as} \quad n \to \infty.\label{012}\end{split}
\end{equation}
Then, the resolvent of the operator $\mathcal{A}$ is not uniformly bounded on the imaginary axis, and consequently, our system is not uniformly (exponentially) stable.  
The proof is thus complete.

\section{Polynomial Stability}\label{polynomialstatrans}
In this section, we will study the polynomial decay of the system \eqref{p5-sys2}-\eqref{p5-diff} for smooth solutions by a multiplier method, since the system \eqref{p5-sys2}-\eqref{p5-diff} is not uniformly stable. One of the main ingredients is to use the frequency domain approach,  specifically the key ingredient for the proof of polynomial stability is Theorem 2.4 of \cite{borichev2010optimal} (see also \cite{batty2008non} and \cite{liu2005characterization}), which we will partially recall:

\begin{Thm}(Borichev-Tomilov (see \cite{borichev2010optimal}))\label{BoriTomi}
   Let $\left( e^{t A} \right)_{t \geq 0}$ be a bounded $C_{0}$-semigroup of contractions on a Hilbert space $H$ generated by $A$ such that $i\mathbb{R}\subset\rho(A)$. Then, for a fixed $\ell>0,$ the following conditions are equivalent:
\begin{itemize}
\item[(i)]:  $\sup\limits_{\beta \in \R}\dfrac{1}{|\beta|^{\ell}}\left\| (i\beta \mathcal{I}-A)^{-1} \right\|_{\mathcal{L}(H)}<+\infty,$
\item[(ii)]: $\text{there exists a constant } C>0~\text{such that for all }U_{0}\in D(A)~\text{we have}\nonumber\\
\|e^{t A} U_{0}\|^2_{H}\leq \dfrac{C}{t^{\frac{2}{\ell}}}\|U_{0}\|^2_{D(A)}, \quad \forall t>0.$
\end{itemize}
\end{Thm}

To proceed with our polynomial energy decay result, we require the following additional geometric assumption on the wave and the plate, which is  due to the requirements of integration in $\Omega_1$ and $\Omega_2$:
\begin{assump}\label{geometriccon}
{\rm Assume that there exists a fixed point $x_0 \in \R^2$ such that, putting $m(x)=x-x_0$, we have 
\begin{equation}\label{intergeo}
(m \cdot \nu_1)=0, \ \ \ \forall x \ \in \ I,
\end{equation}
and
\begin{equation}\label{extgeo}
(m \cdot \nu ) \geq \delta, \ \ \ \forall x \ \in \ \partial \Omega,
\end{equation}
where $\delta$ is a positive real number and $( \ \cdot \ )$ designates the scalar product in $\R^2.$}
\end{assump}
\begin{rk}
 We remark that if condition \eqref{intergeo} is satisfied, then the interface $I$ is straight. Hence, from now on, we assume that the interface is straight and the interior angles $\omega_{i,1}$ and $\omega_{i,N_i}$, for $i=1,2,$ are the angles at the extremities of $I.$
	\end{rk}

The following figures provide some examples of geometries that satisfy the previously mentioned assumptions:
\begin{figure}[hbt!] \begin{subfigure}[b]{0.35\textwidth} \centering \resizebox{\linewidth}{!}{ \begin{tikzpicture}
    \draw[black,-](0,0)--(0,6);
	\draw[blue] (0,0)  -- (-1.5,0) -- (-1.5,6) -- (0,6);
	\draw[red] (0,0) arc (270:450:1.5cm and 3cm);
	\filldraw[black] (0,1.5) circle (2pt) node[right]{$x_0$};
	\filldraw[black] (0,3.5) circle (2pt) node[anchor=west]{};
	\draw[->, ultra thick, black](0,3.5) -- (0,5.5) node[right]{$m$};
	\draw[->, ultra thick, black](0,3.5) -- (1,3.5) node[right]{$\nu_1$};
	\filldraw[red] (1.3,4.5) circle (2pt) node[anchor=west]{};
	\draw[->, ultra thick, red](1.3,4.5) -- (2.09,6.33511) node[right]{$m$};
	\draw[->, ultra thick, red](1.3,4.5) -- (2.25775,4.787327009) node[right]{$\nu_2$};
	\filldraw[blue] (-1.5,4) circle (2pt) node[anchor=west]{};
\draw[->, ultra thick, blue](-1.5,4) -- (-2.4538116355,5.555918769) node[left]{$m$};
\draw[->, ultra thick, blue](-1.5,4) -- (-2.5,4) node[left]{$\nu_1$};
	\node[red,above] at (0.75,3){\scalebox{1}{Plate}};
	\node[red,above] at (0.75,2.5){\scalebox{1}{$\Omega_2$}}; 
		\node[blue,above] at (-0.75,3){\scalebox{1}{ \text{Wave}}}; 
		\node[blue,above] at (-0.75,2.5){\scalebox{1}{ $\Omega_1$}}; 
		\node[black,above] at (-0.25,2){\scalebox{1}{$I$}}; 
				\node[red,right] at (1.5,3){\scalebox{1.25}{$\Gamma_2$}};
					\node[blue,left] at (-1.5,3){\scalebox{1.25}{$\Gamma_1$}};
					\node[blue,above] at (-0.75,6){\scalebox{1.25}{$\Gamma_1$}};
					\node[blue,below] at (-0.75,0){\scalebox{1.25}{$\Gamma_1$}};
	 \end{tikzpicture}} \caption{} \label{fig:subfig8} \end{subfigure} \begin{subfigure}[b]{0.35\textwidth} \centering \resizebox{\linewidth}{!}{ \begin{tikzpicture} 
    \draw[blue] (0.5,2.87) arc (60:300:3);
    \draw[red] (0.5,-2.33) arc (-60:60:3);
    \draw[black,-](0.5,-2.33)--(0.5,2.87);
    \filldraw[black] (0.5,-0.8) circle (2pt) node[right]{$x_0$};
\node[red,above] at (1.3,0){\scalebox{1.4}{Plate}};
	\node[red,above] at (1.3,-0.9){\scalebox{1.4}{$\Omega_2$}}; 
		\node[blue,above] at (-1.5,0){\scalebox{1.5}{ \text{Wave}}}; 
		\node[blue,above] at (-1.5,-0.9){\scalebox{1.5}{ $\Omega_1$}};
		\node[black,above] at (0.2,-0.3){\scalebox{1.5}{$I$}}; 
				\node[red,right] at (2,0){\scalebox{1.25}{$\Gamma_2$}};
					\node[blue,left] at (-4,0){\scalebox{1.25}{$\Gamma_1$}};

    \end{tikzpicture}} \caption{} \label{fig:subfig9} \end{subfigure}\\
\begin{subfigure}[b]{0.35\textwidth} \centering \resizebox{\linewidth}{!}{ \begin{tikzpicture} 
    \draw[blue] (0,1) arc (80:280:2);
    \draw[black,-](0,1)--(0,-2.95);
    \draw[red] (0,1)--(3,-1);
    \draw[red] (0,-2.95)--(3,-1);
    \filldraw[black] (0,-2) circle (2pt) node[right]{$x_0$};
    \node[red,above] at (1,-1){\scalebox{1}{Plate}};
	\node[red,above] at (1,-1.5){\scalebox{1}{$\Omega_2$}}; 
		\node[blue,above] at (-1,-1){\scalebox{1}{ \text{Wave}}}; 
		\node[blue,above] at (-1,-1.5){\scalebox{1}{ $\Omega_1$}};
		\node[black,above] at (-0.15,-1.5){\scalebox{1}{$I$}}; 
				\node[red,right] at (1,0.5){\scalebox{1.25}{$\Gamma_2$}};
				\node[red,right] at (1.1,-2.5){\scalebox{1.25}{$\Gamma_2$}};
					\node[blue,left] at (-2.3,-1){\scalebox{1.25}{$\Gamma_1$}};

    \end{tikzpicture} } \caption{} \label{fig:subfig10} \end{subfigure}
\begin{subfigure}[b]{0.35\textwidth} \centering \resizebox{\linewidth}{!}{ \begin{tikzpicture}
    \draw[blue] (0,0) arc (0:180:3cm and 1.5cm);
    \draw[red] (-6,0) arc (180:360:3cm and 1.5cm);
    \draw[black,-](-6,0)--(0,0);
    \filldraw[black] (-1.5,0) circle (2pt) node[below]{$x_0$};
    \node[black,below] at (-2.5,0){\scalebox{1}{$I$}};
    \node[blue,above] at (-3.5,0.5){\scalebox{1}{ \text{Wave}}}; 
		\node[blue,above] at (-2.5,0.45){\scalebox{1}{ $\Omega_1$}}; 
		\node[red,above] at (-3.5,-0.95){\scalebox{1}{Plate}};
	\node[red,above] at (-2.5,-1){\scalebox{1}{$\Omega_2$}}; 
		\node[blue,right] at (-1.5,1.5){\scalebox{1.25}{$\Gamma_1$}};
					\node[red,left] at (-1,-1.7){\scalebox{1.25}{$\Gamma_2$}}; 
\end{tikzpicture}
 } \caption{} \label{fig:subfig10} \end{subfigure} \caption{Transmission wave-plate models satisfying Assumptions  \ref{anglesomega1}, \ref{anglesomega2} and \ref{geometriccon}}
\label{fig:x cubed graph} \end{figure}
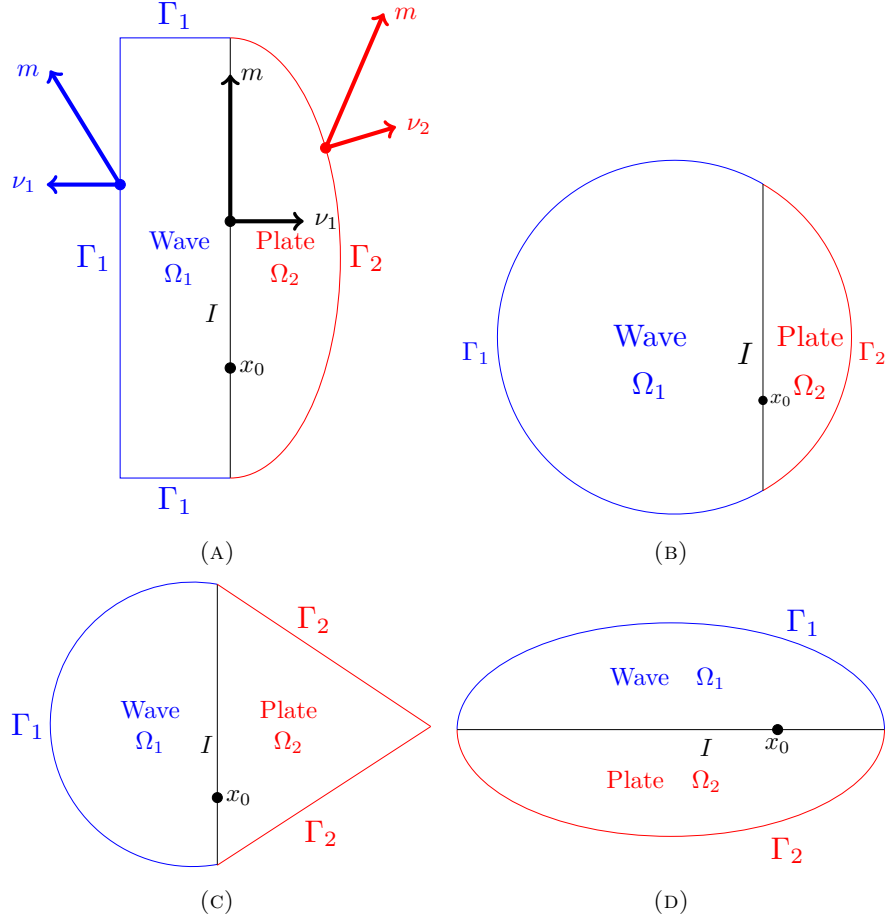

Now, we can state our energy decay rate result.
\begin{Thm}\label{PolynomialStability}
Assume that $\Omega_1$ and $\Omega_2$  are polygonal domains as described above. Also, assume that Assumptions \ref{anglesomega1}, \ref{anglesomega2} and \ref{geometriccon} hold. Then, there exists a constant $\mathcal{C}>0$ such that for all initial data
	$U_0=(u_0,u_1,\eta_0,w_0,w_1,\xi_0,\zeta_0) \in D(\AA),$ the energy of system \eqref{p5-sys2}-\eqref{p5-diff} satisfies the following estimate:
	\begin{equation}\label{ep}
	E(t)\leq \dfrac{\mathcal{C}}{t} \norm{U_0}^2_{D(\mathcal{A})}, \ \   \forall t>0. 
	\end{equation} 
\end{Thm}
	
	Before stating the proof of the above theorem, we shall preliminarily present the following two technical lemmas to be used in the proof of our polynomial-type decay estimate:
\begin{lem}\label{lemuinequality}
		Assume that $\Omega_1$ is a polygonal domain as described above. Assume also that Assumptions \ref{anglesomega1} and \ref{geometriccon} hold. Then, there exists $\varepsilon \in (0,\frac{1}{2})$ such that the solution $y \in H^1(\Omega_1)$ of system \eqref{p5-y1} satisfies the following inequality
		\begin{equation}\label{uinequality}
		-\Re\left\{\int_{\Omega_1} \Delta y (m \cdot \nabla \overline{y}) dx\right\} \geq -\frac{R_1^2}{\delta} \int_{\Gamma_1} |\partial_{\nu_1} y|^2 d\Gamma - \Re\left\{\left\langle \partial_{\nu_1} y,m \cdot \nabla \overline{y}\right\rangle_{H^{-\varepsilon}(I) \times H^{\varepsilon}(I)}\right\}, 
		\end{equation}
		where $R_1=\max\limits_{x \in \Gamma_1}|m(x)|.$
\end{lem}
\begin{proof}
Before starting the proof, we assume that $y \in H^{2}(\Omega_1).$ Using Green's formula, we have the following Rellich's identity (see \cite{komornik1994exact}, also see identity 3.17 of \cite{rao2015stability}):
\begin{equation}\label{Rellich}
-\int_{\Omega_1} \Delta y (m \cdot \nabla \overline{y}) dx=\frac{1}{2} \int_{\partial \Omega_1} (m \cdot \nu_1) |\nabla y|^2 d\Gamma-\int_{\partial \Omega_1} \partial_{\nu_1} y (m \cdot \nabla \overline{y}) d\Gamma.
\end{equation}
According to the geometric condition \eqref{intergeo}, we have from \eqref{Rellich} that 
\begin{equation}\label{Rellichconseq}
-\int_{\Omega_1} \Delta y (m \cdot \nabla \overline{y}) dx=\frac{1}{2} \int_{\Gamma_1} (m \cdot \nu_1) |\nabla y|^2 d\Gamma-\int_{\Gamma_1} \partial_{\nu_1} y (m \cdot \nabla \overline{y}) d\Gamma-\int_{I} \partial_{\nu_1} y (m \cdot \nabla \overline{y}) d\Gamma.
\end{equation}
Next, for any $\varepsilon_1 >0,$ using Young's inequality, we have 
\begin{equation}\label{youngRellich}
\Re\left\{\int_{\Gamma_1} \partial_{\nu_1} y (m \cdot \nabla \overline{y}) d\Gamma\right\} \leq \frac{1}{2\varepsilon_1} \int_{\Gamma_1} |\partial_{\nu_1} y|^2  d\Gamma + \frac{R_1^2 \varepsilon_1}{2} \int_{\Gamma_1} |\nabla y|^2  d\Gamma.
\end{equation}
Using the geometric condition \eqref{extgeo}, we deduce from \eqref{youngRellich} that
\begin{equation}\label{diffRellich}
\begin{split}
\frac{1}{2} \int_{\Gamma_1} (m \cdot \nu_1) |\nabla y|^2 d\Gamma-\Re\left\{\int_{\Gamma_1} \partial_{\nu_1} y (m \cdot \nabla \overline{y}) d\Gamma\right\}\geq \left(\frac{\delta}{2}-\frac{R_1^2 \varepsilon_1}{2}\right)\int_{\Gamma_1} |\nabla y|^2 d\Gamma- \frac{1}{2\varepsilon_1} \int_{\Gamma_1} |\partial_{\nu_1} y|^2  d\Gamma.
\end{split} 
\end{equation}
Choosing $\varepsilon_1=\frac{\delta}{2R_1^2} >0$ in \eqref{diffRellich}, we get 
\begin{equation}\label{epsilonRellich}\begin{split}
\frac{1}{2} \int_{\Gamma_1} (m \cdot \nu_1) |\nabla y|^2 d\Gamma-\Re\left\{\int_{\Gamma_1} \partial_{\nu_1} y (m \cdot \nabla \overline{y}) d\Gamma\right\} &\geq \frac{\delta}{4} \int_{\Gamma_1} |\nabla y|^2 d\Gamma- \frac{R_1^2}{\delta} \int_{\Gamma_1} |\partial_{\nu_1} y|^2  d\Gamma \\ &\geq - \frac{R_1^2}{\delta} \int_{\Gamma_1} |\partial_{\nu_1} y|^2  d\Gamma.\end{split} 
\end{equation}
Taking the real part of \eqref{Rellichconseq} and then inserting \eqref{epsilonRellich}, we obtain 
\begin{equation}\label{finalestu'}\begin{split}
-\Re\left\{\int_{\Omega_1} \Delta y (m \cdot \nabla \overline{y}) dx \right\} \geq - \frac{R_1^2}{\delta} \int_{\Gamma_1} |\partial_{\nu_1} y|^2  d\Gamma-\Re\left\{\int_{I} \partial_{\nu_1} y (m \cdot \nabla \overline{y}) d\Gamma\right\}.\end{split} 
\end{equation}
Now, for $y \in H^1(\Omega_1),$ the solution of \eqref{p5-y1}, we deduce from Lemma \ref{p5-lemregu} and Assumption \ref{anglesomega1} that $y \in H^{\frac{3}{2}-\varepsilon}(\Omega_1)$ for some $\varepsilon \in (0,\frac{1}{2}).$ By density, there exists  a sequence $\left\{y_{k}\right\}_{k \geq 0} \subset H^2(\Omega_1)$ such that 
 \begin{equation*}
y_{k} \underset{k \rightarrow \infty}{\longrightarrow} y \ \ \text{in} \ \ H^{\frac{3}{2}-\varepsilon}(\Omega_1). 
\end{equation*} Applying \eqref{finalestu'} to $y_{k}$ and then passing to the limit in $k,$ we obtain
\begin{equation*}-\Re\left\{\left\langle \Delta y, m \cdot \nabla \overline{y} \right\rangle_{H^{-\frac{1}{2}-\varepsilon}(\Omega_1) \times H^{\frac{1}{2}+\varepsilon}(\Omega_1)} \right\} \geq - \frac{R_1^2}{\delta} \norm{\partial_{\nu_1} y}_{H^{-\varepsilon}(\Gamma_1)}^2  - \Re\left\{\left\langle \partial_{\nu_1} y ,m \cdot \nabla \overline{y} \right\rangle_{H^{-\varepsilon}(I) \times H^{\varepsilon}(I)}\right\}.
\end{equation*} Since $y$ satisfies \eqref{p5-y1}, we see that the duality pairings become integrals in the first two terms of the above identity, which gives \eqref{uinequality}, as desired.
\end{proof}
\begin{lem}\label{lemwinequality}
		Assume that $\Omega_2$ is a polygonal domain as described above. Assume also that Assumptions  \ref{anglesomega2} and \ref{geometriccon} hold. Then, there exists $\varepsilon \in (0,\frac{1}{2})$ such that the solution $y \in H^2(\Omega_2)$ of system \eqref{p5-y2} satisfies the following inequality
		\begin{equation}\begin{split}\label{winequality}
		\Re\left\{\int_{\Omega_2} \Delta^2 y (m \cdot \nabla \overline{y}) dx\right\} \geq \frac{1}{2}a(y,\overline{y})-\frac{\delta(1-\mu)}{4R_2^2}\int_{\Gamma_2} |\partial_{\nu_2} y|^2  d\Gamma-\frac{2R_2^2}{\delta(1-\mu)}\int_{\Gamma_2}|\mathcal{B}_1 y|^2d\Gamma \\-\frac{R_2^2 M}{2}\int_{\Gamma_2}|\mathcal{B}_2 y|^2d\Gamma +\Re\left\{\left\langle \mathcal{B}_2 y , m \cdot \nabla \overline{y} \right\rangle_{H^{-\varepsilon}(I) \times H^{\varepsilon}(I)}\right\},\end{split} 
		\end{equation}
		where $R_2=\max\limits_{x \in \Gamma_2}|m(x)|$ and $M$ is a positive constant.
\end{lem}
\begin{proof}
 Before starting, we assume that $\mathcal{B}_1 y$ and $\mathcal{B}_2 y$ are smooth enough so that $y \in H^4(\Omega_2)$ (see \cite{grisvard1985elliptic}).  Using Green's formula and identity 3.4 of \cite{rao2005polynomial} (see also \cite{lagnese1989boundary} and  \cite{rao1993stabilization}), we have
 \begin{equation}\begin{split}
 \int_{\Omega_2} \Delta^2 y (m \cdot \nabla \overline{y}) dx= a(y,\overline{y})-\int_{\partial \Omega_2 } \left[ \mathcal{B}_1 y \partial_{\nu_2} (m \cdot \nabla \overline{y})-\mathcal{B}_2 y (m \cdot \nabla \overline{y}) \right] d\Gamma + \frac{1}{2} \int_{\partial \Omega_2 } (m \cdot \nu_2) b(y) d\Gamma ,
 \end{split}\end{equation}
 where
 \begin{equation*}
 b(y)=\left(\frac{\partial^2 y}{\partial x_1^2}\right)^2+\left(\frac{\partial^2 y}{\partial x_2^2}\right)^2+ 2\mu \frac{\partial^2 y}{\partial x_1^2}\frac{\partial^2 y}{\partial x_2^2}+2(1-\mu)\left(\frac{\partial^2 y}{\partial x_1 \partial x_2}\right)^2.
 \end{equation*}
 It follows from the geometric condition \eqref{intergeo} and the boundary condition $\mathcal{B}_1 y=0$ on $I$ that 
 \begin{equation}\begin{split}\label{deltasquare}
 \int_{\Omega_2} \Delta^2 y (m \cdot \nabla \overline{y}) dx= a(y,\overline{y})-\int_{\Gamma_2 }  \mathcal{B}_1 y \partial_{\nu_2} (m \cdot \nabla \overline{y}) d\Gamma + \int_{\Gamma_2 } \mathcal{B}_2 y (m \cdot \nabla \overline{y}) d\Gamma \\+ \frac{1}{2} \int_{\Gamma_2 } (m \cdot \nu_2) b(y) d\Gamma +\int_{I} \mathcal{B}_2 y (m \cdot \nabla \overline{y}) d\Gamma.
 \end{split}\end{equation}
 A straightforward computation shows that 
 \begin{equation}\begin{split}\label{bwn}
 b(y)&=\left(\frac{\partial^2 y}{\partial x_1^2}\right)^2+\left(\frac{\partial^2 y}{\partial x_2^2}\right)^2+ 2\mu \frac{\partial^2 y}{\partial x_1^2}\frac{\partial^2 y}{\partial x_2^2}+2(1-\mu)\left(\frac{\partial^2 y}{\partial x_1 \partial x_2}\right)^2 \\ &=\mu \left\{\left(\frac{\partial^2 y}{\partial x_1^2}\right)^2+\left(\frac{\partial^2 y}{\partial x_2^2}\right)^2+2 \frac{\partial^2 y}{\partial x_1^2}\frac{\partial^2 y}{\partial x_2^2}\right\}+(1-\mu)\left\{\left(\frac{\partial^2 y}{\partial x_1^2}\right)^2+\left(\frac{\partial^2 y}{\partial x_2^2}\right)^2+2\left(\frac{\partial^2 y}{\partial x_1 \partial x_2}\right)^2\right\}\\ &=\mu \left(\Delta y\right)^2+(1-\mu)\left\{\left(\frac{\partial^2 y}{\partial x_1^2}\right)^2+\left(\frac{\partial^2 y}{\partial x_2^2}\right)^2+2\left(\frac{\partial^2 y}{\partial x_1 \partial x_2}\right)^2\right\},
 \end{split}\end{equation}
 and 
 \begin{equation}\label{normalinequality}
\left|\partial_{\nu_2} (m \cdot \nabla y)\right| \leq \left|\partial_{\nu_2} y\right|+R_2\left\{\left(\frac{\partial^2 y}{\partial x_1^2}\right)^2+\left(\frac{\partial^2 y}{\partial x_2^2}\right)^2+2\left(\frac{\partial^2 y}{\partial x_1 \partial x_2}\right)^2 \right\}^{\frac{1}{2}}.
 \end{equation}
 Indeed, applying Young's inequality, we obtain from \eqref{normalinequality} that
 \begin{equation}\begin{split}\label{B1inequality}
 \Re\left\{\int_{\Gamma_2 }  \mathcal{B}_1 y \partial_{\nu_2} (m \cdot \nabla \overline{y}) d\Gamma\right\} &\leq \frac{1}{2\varepsilon_2} \int_{\Gamma_2} |\mathcal{B}_1 y|^2  d\Gamma + \frac{ \varepsilon_2}{2} \int_{\Gamma_2} \left| \partial_{\nu_2} (m \cdot \nabla \overline{y})\right|^2 d\Gamma \\ &\leq \frac{1}{2\varepsilon_2} \int_{\Gamma_2} |\mathcal{B}_1 y|^2  d\Gamma + \varepsilon_2 \int_{\Gamma_2} \left|\partial_{\nu_2} y\right|^2  d\Gamma \\&+ R_2^2 \varepsilon_2  \int_{\Gamma_2} \left\{\left(\frac{\partial^2 y}{\partial x_1^2}\right)^2+\left(\frac{\partial^2 y}{\partial x_2^2}\right)^2+2\left(\frac{\partial^2 y}{\partial x_1 \partial x_2}\right)^2 \right\} d\Gamma,
 \end{split}\end{equation}
 for any $\varepsilon_2>0.$ Taking the geometric condition \eqref{extgeo} into consideration, it follows from \eqref{bwn} that
 \begin{equation}\begin{split}\label{bwninequality}
\frac{1}{2} \int_{\Gamma_2 } (m \cdot \nu_2) b(y) d\Gamma & \geq \frac{\delta}{2} \int_{\Gamma_2 }b(y) d\Gamma \\ & \geq \frac{\delta \mu }{2} \int_{\Gamma_2 }\left(\Delta y\right)^2 d\Gamma  + \frac{\delta (1-\mu)}{2} \int_{\Gamma_2 }\left\{\left(\frac{\partial^2 y}{\partial x_1^2}\right)^2+\left(\frac{\partial^2 y}{\partial x_2^2}\right)^2+2\left(\frac{\partial^2 y}{\partial x_1 \partial x_2}\right)^2 \right\} d\Gamma \\ & \geq \frac{\delta (1-\mu)}{2} \int_{\Gamma_2 }\left\{\left(\frac{\partial^2 y}{\partial x_1^2}\right)^2+\left(\frac{\partial^2 y}{\partial x_2^2}\right)^2+2\left(\frac{\partial^2 y}{\partial x_1 \partial x_2}\right)^2 \right\} d\Gamma.
\end{split} \end{equation}
Combining \eqref{B1inequality} and \eqref{bwninequality}, we obtain 
\begin{equation}\begin{split}\label{B1&bwn}
\frac{1}{2} \int_{\Gamma_2 } (m \cdot \nu_2) b(y) d\Gamma-\Re\left\{\int_{\Gamma_2 }  \mathcal{B}_1 y \partial_{\nu_2} (m \cdot \nabla \overline{y}) d\Gamma\right\} \geq -\frac{1}{2\varepsilon_2} \int_{\Gamma_2} |\mathcal{B}_1 y|^2  d\Gamma - \varepsilon_2 \int_{\Gamma_2} \left|\partial_{\nu_2} y\right|^2  d\Gamma \\+ \left(\frac{\delta (1-\mu)}{2}-R_2^2 \varepsilon_2 \right)\int_{\Gamma_2 }\left\{\left(\frac{\partial^2 y}{\partial x_1^2}\right)^2+\left(\frac{\partial^2 y}{\partial x_2^2}\right)^2+2\left(\frac{\partial^2 y}{\partial x_1 \partial x_2}\right)^2 \right\} d\Gamma.
\end{split}\end{equation}
By taking $\varepsilon_2=\frac{\delta (1-\mu)}{4 R_2^2} > 0$ in \eqref{B1&bwn}, we deduce
\begin{equation*}\begin{split}
\frac{1}{2} \int_{\Gamma_2 } (m \cdot \nu_2) b(y) d\Gamma-\Re\left\{\int_{\Gamma_2 }  \mathcal{B}_1 y \partial_{\nu_2} (m \cdot \nabla \overline{y}) d\Gamma\right\} &\geq -\frac{2R_2^2}{\delta (1-\mu)} \int_{\Gamma_2} |\mathcal{B}_1 y|^2  d\Gamma - \frac{\delta (1-\mu)}{4 R_2^2} \int_{\Gamma_2} \left|\partial_{\nu_2} y\right|^2  d\Gamma \\&+ \frac{\delta (1-\mu)}{4} \int_{\Gamma_2 }\left\{\left(\frac{\partial^2 y}{\partial x_1^2}\right)^2+\left(\frac{\partial^2 y}{\partial x_2^2}\right)^2+2\left(\frac{\partial^2 y}{\partial x_1 \partial x_2}\right)^2 \right\} d\Gamma,\end{split}
\end{equation*}
from which we conclude
\begin{equation}\begin{split}\label{B1&bwnfinal}
\frac{1}{2} \int_{\Gamma_2 } (m \cdot \nu_2) b(y) d\Gamma-\Re\left\{\int_{\Gamma_2 }  \mathcal{B}_1 y \partial_{\nu_2} (m \cdot \nabla \overline{y}) d\Gamma\right\} \geq -\frac{2R_2^2}{\delta (1-\mu)} \int_{\Gamma_2} |\mathcal{B}_1 y|^2  d\Gamma \\ - \frac{\delta (1-\mu)}{4 R_2^2} \int_{\Gamma_2} \left|\partial_{\nu_2} y\right|^2  d\Gamma.
\end{split}\end{equation}
On the other hand, applying Young's inequality once more gives
\begin{equation}\label{B2inequality}
\Re \left\{\int_{\Gamma_2 } \mathcal{B}_2 y (m \cdot \nabla \overline{y}) d\Gamma\right\} \geq - \frac{1}{2\varepsilon_3} \int_{\Gamma_2} |\mathcal{B}_2 y|^2  d\Gamma - \frac{ R_2^2 \varepsilon_3 }{2} \int_{\Gamma_2} \left| \nabla y\right|^2 d\Gamma,
\end{equation}
for any $\varepsilon_3 >0.$ Thanks to the trace theorem (see Theorem 1.6.6 of \cite{brenner2008mathematical}), we know that there exists a constant $C_1 >0$ such that 
\begin{equation}\label{traceinequality}
\norm{\nabla y}_{L^2(\Gamma_2)}^2 \leq C_1 \norm{\nabla y}_{L^2(\Omega_2)} \norm{\nabla y}_{H^1(\Omega_2)} \leq  C_1 \norm{y}_{H^2(\Omega_2)}^2.
\end{equation}
Moreover, according to the fact that $a(y,\overline{y})$ is equivalent to the usual norm of $H^2(\Omega_2)$ on $H_{\ast}^2 (\Omega_2)$, we know that there exists a constant $C_2 >0$ such that
\begin{equation}\label{equivalenceinequality}
\norm{y}_{H^2(\Omega_2)}^2 \leq C_2 a(y,\overline{y}).
\end{equation}
The combination of \eqref{traceinequality} and \eqref{equivalenceinequality} yields 
\begin{equation}\label{trace&equiv}
\norm{\nabla y}_{L^2(\Gamma_2)}^2 \leq M a(y,\overline{y}),
\end{equation}
where the constant $M=C_1 C_2 >0.$ This, along with \eqref{B2inequality}, implies that
\begin{equation}\label{B2inequality&trace&equiv}
\Re \left\{\int_{\Gamma_2 } \mathcal{B}_2 y (m \cdot \nabla \overline{y}) d\Gamma\right\} \geq - \frac{1}{2\varepsilon_3} \int_{\Gamma_2} |\mathcal{B}_2 y|^2  d\Gamma - \frac{ R_2^2 M  \varepsilon_3 }{2} a(y,\overline{y}).
\end{equation}
By taking the real part of \eqref{deltasquare} and inserting \eqref{B1&bwnfinal} and \eqref{B2inequality&trace&equiv}, we obtain
\begin{equation}\begin{split}\label{deltasquare&B1&bwnfinal&B2}
\Re\left\{\int_{\Omega_2} \Delta^2 y (m \cdot \nabla \overline{y}) dx\right\} \geq \left(1-\frac{ R_2^2 M  \varepsilon_3 }{2}\right) a(y,\overline{y}) -\frac{2R_2^2}{\delta (1-\mu)} \int_{\Gamma_2} |\mathcal{B}_1 y|^2  d\Gamma  - \frac{\delta (1-\mu)}{4 R_2^2} \int_{\Gamma_2} \left|\partial_{\nu_2} y\right|^2  d\Gamma \\- \frac{1}{2\varepsilon_3} \int_{\Gamma_2} |\mathcal{B}_2 y|^2  d\Gamma + \Re\left\{\int_{I} \mathcal{B}_2 y (m \cdot \nabla \overline{y}) d\Gamma\right\}. \end{split}
\end{equation}
By letting $\varepsilon_3 =\frac{1}{R_2^2 M} >0$ in \eqref{deltasquare&B1&bwnfinal&B2}, we conclude that 
\begin{equation}\label{finalestw'}\begin{split}
\Re\left\{\int_{\Omega_2} \Delta^2 y (m \cdot \nabla \overline{y}) dx\right\} \geq \frac{1}{2} a(y,\overline{y}) -\frac{2R_2^2}{\delta (1-\mu)} \int_{\Gamma_2} |\mathcal{B}_1 y|^2  d\Gamma - \frac{\delta (1-\mu)}{4 R_2^2} \int_{\Gamma_2} \left|\partial_{\nu_2} y\right|^2  d\Gamma \\- \frac{R_2^2 M}{2} \int_{\Gamma_2} |\mathcal{B}_2 y|^2  d\Gamma + \Re\left\{\int_{I} \mathcal{B}_2 y (m \cdot \nabla \overline{y}) d\Gamma\right\}. \end{split}
\end{equation}
Now, for $y \in H^2(\Omega_2),$ which is a solution of \eqref{p5-y2}, we deduce, based on Lemma \ref{p5-lemregw} and Assumption  \ref{anglesomega2} that there exists a sequence $\left\{y_{k}\right\}_{k \geq 0} \subset H^4(\Omega_2)$ such that \begin{equation*}
y_{k} \underset{k \rightarrow \infty}{\longrightarrow} y \ \ \text{in} \ \ H^2(\Omega_2) \ \ \ \text{and} \ \ \ \Delta^2 y_{k} \underset{k \rightarrow \infty}{\longrightarrow} \Delta^2 y \ \ \text{in} \ \ L^2(\Omega_2). 
\end{equation*} By applying \eqref{finalestw'} to $y_{k}$ and passing the limit in $k,$ we obtain the desired estimation \eqref{winequality}.
\end{proof}
\\

We are now ready to present the proof of Theorem \ref{PolynomialStability}.\\

\textbf{Proof.} In accordance to Theorem \ref{BoriTomi}, a $C_{0}$-semigroup of contractions $(e^{t\mathcal{A}})_{t \geq 0 }$ on a Hilbert space $\HH$ verifies \eqref{ep} if the following conditions 
\begin{align*}
i \R \subset \rho(\mathcal{A}) &\qquad  (P1)
\end{align*} and \begin{align*} 
\lim\sup_{ |\lambda| \to \infty} \dfrac{1}{\lambda ^{2}} \left\| (i\lambda \mathcal{I}-\mathcal{A})^{-1} \right\|_{\mathcal{L}(\HH)}<+\infty &\qquad (P2)
\end{align*}
are satisfied. Since the resolvent of the operator $\AA$ is not compact in the energy space $\HH$ and $0 \in \rho(\mathcal{A})$, proving $ i \R \subset \rho(\mathcal{A})$ is equivalent to showing that $i \lambda \mathcal{I} - \AA$ is bijective in the energy space $\HH$ for all $\lambda \in \R^*.$  This is established in Section \ref{StrongSta} using a unique continuation theorem and Fredholm's alternative. Then, we still need to prove condition $(P2).$ This is checked by a contradiction argument. Indeed, suppose that there exists a sequence $\left\{\lambda_n\right\}_{n\geq 1} \subset \R^*_+$ and a sequence $\{U_n:=(u_n,v_n,\eta_n,w_n,z_n,\xi_n,\zeta_n)\}_{n \geq 1} \subset D(\AA)$ such that 
\begin{equation}\label{c1}
	\lambda_n \longrightarrow+\infty,\qquad \norm{U_n}_{\HH}=1,
\end{equation}
and 
\begin{equation}\label{c2}
	\lambda_n^2(i\lambda_n \mathcal{I}-\AA)U_n=(f_n^1,g_n^1,h_n^1,f_n^2,g_n^2,h_n^2,h_n^3)\longrightarrow 0\quad {\rm in}\  \HH.
\end{equation}
Our aim is to show that $\norm{(u_n,v_n,\eta_n,w_n,z_n,\xi_n,\zeta_n)}_{\HH} \longrightarrow 0.$ This condition provides a contradiction with \eqref{c1}. By detailing equation \eqref{c2}, we get the following system
\begin{eqnarray}
	\lambda_n^2 (i\lambda_n u_n-v_n) &=& f_n^1 \,\,\rightarrow \,0 \hskip 0.5 cm \mbox{in}\hskip 0.5 cm H_*^1(\Omega_1),\label{eq1}\\
	\lambda_n^2 (i\lambda_n v_n- \Delta u_n)&=& g_n^1 \,\,\rightarrow \,0 \hskip 0.5 cm \mbox{in}\hskip 0.5 cm L^2(\Omega_1),\label{eq2}\\
	\lambda_n^2 (i\lambda_n \eta_n-v_n+ \eta_n) &=& h_n^1 \,\,\rightarrow \,0 \hskip 0.5 cm \mbox{in}\hskip 0.5 cm L^2(\Gamma_1),\label{eq3}\\
	\lambda_n^2 (i\lambda_n w_n- z_n)  &=& f_n^2 \,\,\rightarrow
	\,0 \hskip 0.5 cm \mbox{in}\hskip 0.5 cm H_*^2(\Omega_2),\label{eq4}\\
	\lambda_n^2 (i\lambda_n z_n+ \Delta^2 w_n)  &=& g_n^2 \,\,\rightarrow
	\,0 \hskip 0.5 cm \mbox{in}\hskip 0.5 cm L^2(\Omega_2),\label{eq5}\\
	\lambda_n^2 (i\lambda_n \xi_n- \partial_{\nu_2}z_n + \xi_n)  &=& h_n^2 \,\,\rightarrow
	\,0 \hskip 0.5 cm \mbox{in}\hskip 0.5 cm L^2(\Gamma_2),\label{eq6}\\
	\lambda_n^2 (i\lambda_n \zeta_n- z_n+ \zeta_n) &=& h_n^3 \,\,\rightarrow
	\,0 \hskip 0.5 cm \mbox{in}\hskip 0.5 cm L^2(\Gamma_2).\label{eq7}
\end{eqnarray}
We notice from \eqref{c1} that $v_n$ and $z_n$ are uniformly bounded in $L^2(\Omega_1)$ and $L^2(\Omega_2),$ respectively. It follows from Equations
\eqref{eq1} and \eqref{eq4} that
\begin{equation}\label{esbdd}
\norm{u_n}_{L^2(\Omega_1)}=\frac{O(1)}{\lambda_n} \qquad \text{and} \qquad \norm{w_n}_{L^2(\Omega_2)}=\frac{O(1)}{\lambda_n}.
\end{equation}
On the other hand, inserting Equation \eqref{eq1} (resp. \eqref{eq4}) into Equation \eqref{eq2} (resp. \eqref{eq5}), we obtain the following system
\begin{eqnarray}
	-\lambda_n^2  u_n-\Delta u_n &=& \frac{i f_n^1}{\lambda_n}+\frac{g_n^1}{\lambda_n^2},\label{eq1inserteq2}\\
	-\lambda_n^2  w_n+\Delta^2 w_n &=& \frac{i f_n^2}{\lambda_n}+\frac{g_n^2}{\lambda_n^2}.\label{eq4inserteq5}
\end{eqnarray}
For clarity, the proof is divided into several lemmas:
\begin{lem}\label{p5-diss}
		The solution $(u_n,v_n,\eta_n,w_n,z_n,\xi_n,\zeta_n) \in D(\AA)$ of system \eqref{eq1}-\eqref{eq7} satisfies the following asymptotic behavior estimation
		\begin{equation}\label{esdiss}
		\int_{\Gamma_1} |\eta_n|^2 d\Gamma=\frac{o(1)}{\lambda_n^2}, \quad  \int_{\Gamma_2} |\xi_n|^2 d\Gamma=\frac{o(1)}{\lambda_n^2} \quad \text{and} \quad \int_{\Gamma_2} |\zeta_n|^2 d\Gamma=\frac{o(1)}{\lambda_n^2}.
		\end{equation}
\end{lem}
\begin{proof}
Taking the inner product of \eqref{c2} with $U_n=(u_n,v_n,\eta_n,w_n,z_n,\xi_n,\zeta_n)$ in $\HH$, then using the fact that $U_n$ is uniformly bounded in $\HH$, we get 
\begin{equation*}
 \int_{\Gamma_1} |\eta_n|^2 d\Gamma + \int_{\Gamma_2} |\xi_n|^2 d\Gamma + \int_{\Gamma_2} |\zeta_n|^2 d\Gamma= \Re\{\left( (i\lambda_n \mathcal{I}-\AA)U_n,U_n\right)_{\HH}\}=\frac{o(1)}{\lambda_n^2},
\end{equation*}
which implies that 
\begin{equation*}
\int_{\Gamma_1} |\eta_n|^2 d\Gamma=\frac{o(1)}{\lambda_n^2}, \qquad  \int_{\Gamma_2} |\xi_n|^2 d\Gamma=\frac{o(1)}{\lambda_n^2} \qquad \text{and} \qquad \int_{\Gamma_2} |\zeta_n|^2 d\Gamma=\frac{o(1)}{\lambda_n^2}.
\end{equation*}
\end{proof}
\begin{lem}
		The solution $(u_n,v_n,\eta_n,w_n,z_n,\xi_n,\zeta_n) \in D(\AA)$ of system \eqref{eq1}-\eqref{eq7} satisfies the following asymptotic behavior estimation
		\begin{eqnarray}
		\norm{u_n}_{L^2(\Gamma_1)}&=&\frac{o(1)}{\lambda_n},\label{dissconseq1}\\
		\norm{\partial_{\nu_2}w_n}_{L^2(\Gamma_2)}&=&\frac{o(1)}{\lambda_n},\label{dissconseq2}\\
		\norm{w_n}_{L^2(\Gamma_2)}&=&\frac{o(1)}{\lambda_n}.\label{dissconseq3}
		\end{eqnarray}
\end{lem}
\begin{proof}
First, using Equation \eqref{eq3} and the first estimation of \eqref{esdiss}, we get 
\begin{equation}\label{esv}
\norm{v_n}_{L^2(\Gamma_1)}=o(1).
\end{equation}
From Equation \eqref{eq1} and the fact that the trace operator $\varphi \longmapsto \varphi|_{\Gamma_1}$ is a linear and continuous mapping from $H^1_*(\Omega_1)$ into $L^2(\Gamma_1),$ we have 
\begin{equation*}
i\lambda_n u_n-v_n = \frac{f_n^1}{\lambda_n^2}  \,\,\rightarrow \,0 \hskip 0.5 cm \mbox{in}\hskip 0.5 cm L^2(\Gamma_1).
\end{equation*}
Multiplying the above equation by $-i \lambda_n \overline{u_n}$ and integrating over $\Gamma_1$, we obtain
\begin{equation*}
\int_{\Gamma_1}|\lambda_n u_n|^2 d\Gamma + \int_{\Gamma_1}i \lambda_n \overline{u_n} v_n d\Gamma =-\int_{\Gamma_1}\frac{i f_n^1 \overline{u_n}}{\lambda_n}  d\Gamma.
\end{equation*}
Using estimations \eqref{esbdd}, \eqref{esv} and the fact that $f_n^1$ converges to zero in $L^2(\Gamma_1),$ we deduce from the above equation that
\begin{equation*}
\int_{\Gamma_1}|\lambda_n u_n|^2 d\Gamma=o(1),
\end{equation*}
which gives the desired estimation \eqref{dissconseq1}.
Next, using Equations \eqref{eq6}, \eqref{eq7} and the second two estimations of \eqref{esdiss}, we obtain
\begin{equation}\label{esz}
\norm{\partial_{\nu_2}z_n}_{L^2(\Gamma_2)}=o(1) \quad \text{and} \quad \norm{z_n}_{L^2(\Gamma_2)}=o(1).
\end{equation}
Similar computation performed on Equation \eqref{eq4}, using estimation \eqref{esz} and the fact that $f_n^2$  converges to zero in $L^2(\Gamma_2)$, yield, as well,
\begin{equation*}
\int_{\Gamma_2}|\lambda_n \partial_{\nu_2} w_n|^2 d\Gamma=o(1) \quad \text{and} \quad \int_{\Gamma_2}|\lambda_n w_n|^2 d\Gamma=o(1),
\end{equation*}
which gives the desired estimations \eqref{dissconseq2} and \eqref{dissconseq3}.
\end{proof}
\begin{lem}
		The solution $(u_n,v_n,\eta_n,w_n,z_n,\xi_n,\zeta_n) \in D(\AA)$ of system \eqref{eq1}-\eqref{eq7} satisfies the following asymptotic behavior estimation
		\begin{eqnarray}
		\norm{\partial_{\nu_1} u_n}_{L^2(\Gamma_1)}&=&\frac{o(1)}{\lambda_n},\label{domainconseq1}\\
		\norm{\mathcal{B}_1 w_n}_{L^2(\Gamma_2)}&=&\frac{o(1)}{\lambda_n},\label{domainconseq2}\\
		\norm{\mathcal{B}_2 w_n}_{L^2(\Gamma_2)}&=&\frac{o(1)}{\lambda_n}.\label{domainconseq3}
		\end{eqnarray}
\end{lem}
\begin{proof}
From the boundary conditions of \eqref{p5-domain}, we have $\partial_{\nu_1} u_n=-\eta_n$ on $\Gamma_1$, $\mathcal{B}_1 w_n=-\xi_n$ and $\mathcal{B}_2 w_n=\zeta_n$ on $\Gamma_2.$ Therefore, using estimation \eqref{esdiss}, the desired asymptotic estimations \eqref{domainconseq1}, \eqref{domainconseq2} and \eqref{domainconseq3} follow.
\end{proof}
\begin{lem}\label{lemu&winequality}
		Assume that Assumptions \ref{anglesomega1}, \ref{anglesomega2} and \ref{geometriccon} hold. Then, the solution $(u_n,v_n,\eta_n,w_n,z_n,\xi_n,\zeta_n) \in D(\AA)$ of system \eqref{eq1}-\eqref{eq7} satisfies the following estimation
		\begin{equation}\begin{split}\label{u&winequality}
		\Re\left\{\int_{\Omega_2} \Delta^2 w_n (m \cdot \nabla \overline{w_n}) dx -\int_{\Omega_1} \Delta u_n (m \cdot \nabla \overline{u_n}) dx\right\} \geq  \frac{1}{2} a(w_n,\overline{w_n})-\frac{R_1^2}{\delta} \int_{\Gamma_1} |\partial_{\nu_1} u_n|^2 d\Gamma \\ -\frac{\delta(1-\mu)}{4R_2^2}\int_{\Gamma_2} |\partial_{\nu_2} w_n|^2  d\Gamma -\frac{2R_2^2}{\delta(1-\mu)}\int_{\Gamma_2}|\mathcal{B}_1 w_n|^2d\Gamma -\frac{R_2^2 M}{2}\int_{\Gamma_2}|\mathcal{B}_2 w_n|^2d\Gamma, \end{split}
		\end{equation}
		where $R_1=\max\limits_{x \in \Gamma_1}|m(x)|,$ $R_2=\max\limits_{x \in \Gamma_2}|m(x)|$ and $M$ is a positive constant independent of $n.$
\end{lem}
\begin{proof}
By the definition of $D(\AA),$ $u_n$ may be seen as the unique solution $u_n \in H^1(\Omega_1)$ of system \eqref{p5-y1} with $v_1=-\eta_n$ and $v_2=w_n.$ As Assumptions \ref{anglesomega1} and \ref{geometriccon} hold, it suffices to apply Lemma \ref{lemuinequality} for $u_n$ to obtain
\begin{equation}\label{unninequality}
		-\Re\left\{\int_{\Omega_1} \Delta u_n (m \cdot \nabla \overline{u_n}) dx\right\} \geq -\frac{R_1^2}{\delta} \int_{\Gamma_1} |\partial_{\nu_1} u_n|^2 d\Gamma - \Re\left\{\left\langle \partial_{\nu_1} u_n,m \cdot \nabla \overline{u_n}\right\rangle_{H^{-\varepsilon}(I) \times H^{\varepsilon}(I)}\right\}, 
		\end{equation}
for some $\varepsilon \in (0,\frac{1}{2}).$ Again, since $U_n \in D(\AA),$ we see that $w_n \in H^2(\Omega_2)$ is a solution of system \eqref{p5-y2} with $v_3=\partial_{\nu_1} u_n,$ $v_4=-\xi_n$ and $v_5=\zeta_n.$ It follows from Lemma \ref{lemwinequality}, Assumptions \ref{anglesomega2} and \ref{geometriccon} that
\begin{equation}\begin{split}\label{wnninequality}
		\Re\left\{\int_{\Omega_2} \Delta^2 w_n (m \cdot \nabla \overline{w_n}) dx\right\} \geq \frac{1}{2}a(w_n,\overline{w_n})-\frac{\delta(1-\mu)}{4R_2^2}\int_{\Gamma_2} |\partial_{\nu_2} w_n|^2  d\Gamma-\frac{2R_2^2}{\delta(1-\mu)}\int_{\Gamma_2}|\mathcal{B}_1 w_n|^2d\Gamma \\ -\frac{R_2^2 M}{2}\int_{\Gamma_2}|\mathcal{B}_2 w_n|^2d\Gamma+\Re\left\{\left\langle \mathcal{B}_2 w_n , m \cdot \nabla \overline{w_n} \right\rangle_{H^{-\varepsilon}(I) \times H^{\varepsilon}(I)}\right\}.\end{split} 
		\end{equation} Now, by adding \eqref{unninequality} and \eqref{wnninequality}, we obtain
\begin{equation}\begin{split}\label{proofu&winequality}
		\Re\left\{\int_{\Omega_2} \Delta^2 w_n (m \cdot \nabla \overline{w_n}) dx\right\}-\Re\left\{\int_{\Omega_1} \Delta u_n (m \cdot \nabla \overline{u_n}) dx\right\} \geq  \frac{1}{2} a(w_n,\overline{w_n})-\frac{R_1^2}{\delta} \int_{\Gamma_1} |\partial_{\nu_1} u_n|^2 d\Gamma \\ -\frac{\delta(1-\mu)}{4R_2^2}\int_{\Gamma_2} |\partial_{\nu_2} w_n|^2  d\Gamma -\frac{2R_2^2}{\delta(1-\mu)}\int_{\Gamma_2}|\mathcal{B}_1 w_n|^2d\Gamma -\frac{R_2^2 M}{2}\int_{\Gamma_2}|\mathcal{B}_2 w_n|^2d\Gamma \\ + \Re\left\{\left\langle \mathcal{B}_2 w_n , m \cdot \nabla \overline{w_n} \right\rangle_{H^{-\varepsilon}(I) \times H^{\varepsilon}(I)}-\left\langle \partial_{\nu_1} u_n,m \cdot \nabla \overline{u_n}\right\rangle_{H^{-\varepsilon}(I) \times H^{\varepsilon}(I)}\right\}. \end{split}
		\end{equation} The boundary conditions $u_n=w_n$ and $\mathcal{B}_2 w_n=\partial_{\nu_1} u_n$ on the interface $I,$ along with the geometric condition \eqref{intergeo}, lead to 
		\begin{equation*}\begin{split}
&\left\langle \mathcal{B}_2 w_n , m \cdot \nabla \overline{w_n} \right\rangle_{H^{-\varepsilon}(I) \times H^{\varepsilon}(I)}-\left\langle \partial_{\nu_1} u_n,m \cdot \nabla \overline{u_n}\right\rangle_{H^{-\varepsilon}(I) \times H^{\varepsilon}(I)} \\ \\ &=\left\langle \mathcal{B}_2 w_n, m \cdot (\nu_2 \partial_{\nu_2} \overline{w_n} +\tau_2 \partial_{\tau_2} \overline{w_n}) \right\rangle_{H^{-\varepsilon}(I) \times H^{\varepsilon}(I)}-\left\langle \partial_{\nu_1} u_n, m \cdot (\nu_1 \partial_{\nu_1} \overline{u_n} +\tau_1 \partial_{\tau_1} \overline{u_n}) \right\rangle_{H^{-\varepsilon}(I) \times H^{\varepsilon}(I)} \\ \\ &=	\left\langle \mathcal{B}_2 w_n, m \cdot \tau_2 \partial_{\tau_2} \overline{w_n} \right\rangle_{H^{-\varepsilon}(I) \times H^{\varepsilon}(I)}-\left\langle \partial_{\nu_1} u_n, m \cdot \tau_1 \partial_{\tau_1} \overline{u_n} \right\rangle_{H^{-\varepsilon}(I) \times H^{\varepsilon}(I)} \\ \\ &=\left\langle \mathcal{B}_2 w_n, m \cdot \tau_2 \partial_{\tau_2} \overline{w_n} \right\rangle_{H^{-\varepsilon}(I) \times H^{\varepsilon}(I)}-\left\langle \partial_{\nu_1} u_n,  m \cdot \tau_2 \partial_{\tau_2} \overline{w_n} \right\rangle_{H^{-\varepsilon}(I) \times H^{\varepsilon}(I)} \\ \\ &=\left\langle \mathcal{B}_2 w_n-\partial_{\nu_1} u_n, m \cdot \tau_2 \partial_{\tau_2} \overline{w_n} \right\rangle_{H^{-\varepsilon}(I) \times H^{\varepsilon}(I)}=0.\end{split}
		\end{equation*}
		This, together with \eqref{proofu&winequality}, proves \eqref{u&winequality}, as desired.
\end{proof}
\begin{lem}
		Assume that Assumptions \ref{anglesomega1} and  \ref{anglesomega2} hold. Then, the solution $(u_n,v_n,\eta_n,w_n,z_n,\xi_n,\zeta_n) \in D(\AA)$ of system \eqref{eq1}-\eqref{eq7} satisfies the following asymptotic behavior estimation
		\begin{equation}\label{u&wequiv}
\int_{\Omega_1} \left|\lambda_n u_n\right|^2 dx + \int_{\Omega_2} \left|\lambda_n w_n\right|^2  dx = 	\int_{\Omega_1} \left|\nabla u_n\right|^2 dx +a(w_n,\overline{w_n})+\frac{o(1)}{\lambda_n^2}.
		\end{equation}
\end{lem}
\begin{proof}
Before starting the proof, we assume that $u_n \in H^2(\Omega_1)$ and $w_n \in H^4(\Omega_2).$ Multiplying Equation \eqref{eq1inserteq2} (resp. \eqref{eq4inserteq5}) by $-\overline{u_n}$ (resp. $-\overline{w_n}$), integrating over $\Omega_1$ (resp. $\Omega_2$) and applying Green's formula, we obtain 
\begin{equation}
	\int_{\Omega_1} \left|\lambda_n  u_n\right|^2 dx -\int_{\Omega_1} \left|\nabla  u_n\right|^2 dx + \int_{\partial \Omega_1} \partial_{\nu_1} u_n \overline{u_n} d\Gamma = - \int_{\Omega_1} \left(\frac{i f_n^1}{\lambda_n}+\frac{g_n^1}{\lambda_n^2}\right) \overline{u_n} dx,\label{eq1inserteq2mult}
	\end{equation}
	and
	\begin{equation}
	\int_{\Omega_2} \left|\lambda_n  w_n\right|^2 dx -a(w_n,\overline{w_n}) + \int_{\partial \Omega_2} \left( \mathcal{B}_1 w_n \partial_{\nu_2}\overline{w_n}-\mathcal{B}_2 w_n \overline{w_n}   \right) d\Gamma = - \int_{\Omega_2} \left(\frac{i f_n^2}{\lambda_n}+\frac{g_n^2}{\lambda_n^2}\right) \overline{w_n} dx.\label{eq4inserteq5mult}
	\end{equation}
	Taking into consideration that $f_n^1,$ $g_n^1$ converge to zero in $L^2(\Omega_1),$ $f_n^2,$ $g_n^2$ converge to zero in $L^2(\Omega_2)$ and using \eqref{esbdd}, we get from \eqref{eq1inserteq2mult} and \eqref{eq4inserteq5mult} that  
	\begin{equation*}
	\int_{\Omega_1} \left|\lambda_n  u_n\right|^2 dx -\int_{\Omega_1} \left|\nabla  u_n\right|^2 dx + \int_{\partial \Omega_1} \partial_{\nu_1} u_n \overline{u_n} d\Gamma = \frac{o(1)}{\lambda_n^2},
	\end{equation*}
	and
	\begin{equation*}
	\int_{\Omega_2} \left|\lambda_n  w_n\right|^2 dx -a(w_n,\overline{w_n}) + \int_{\partial \Omega_2} \left( \mathcal{B}_1 w_n \partial_{\nu_2}\overline{w_n}-\mathcal{B}_2 w_n \overline{w_n}   \right) d\Gamma = \frac{o(1)}{\lambda_n^2}.
	\end{equation*}
	Adding both equations and taking into account the boundary conditions on the interface, as well as the estimations \eqref{dissconseq1}, \eqref{dissconseq2}, \eqref{dissconseq3}, \eqref{domainconseq1}, \eqref{domainconseq2} and \eqref{domainconseq3}, we obtain 
	\begin{equation*}
\int_{\Omega_1} \left|\lambda_n u_n\right|^2 dx - \int_{\Omega_1} \left|\nabla u_n\right|^2 dx + \int_{\Omega_2} \left|\lambda_n w_n\right|^2  dx - a(w_n,\overline{w_n})=\frac{o(1)}{\lambda_n^2}.
		\end{equation*}
By following a similar argument to the end of the proof of Lemma \ref{lemuinequality} and Lemma \ref{lemwinequality}, we obtain the desired estimation \eqref{u&wequiv}.
\end{proof}
\begin{lem}
		Assume that Assumptions \ref{anglesomega1}, \ref{anglesomega2} and \ref{geometriccon} hold. Then, the solution $(u_n,v_n,\eta_n,w_n,z_n,\xi_n,\zeta_n) \in D(\AA)$ of system \eqref{eq1}-\eqref{eq7} satisfies the following asymptotic behavior estimation
		\begin{equation}\label{u&w&aww}
\int_{\Omega_1} \left|\lambda_n u_n\right|^2 dx=o(1), \quad  \int_{\Omega_2} \left|\lambda_n w_n\right|^2  dx=o(1) \quad \text{and} \quad a(w_n,\overline{w_n})=o(1).
		\end{equation}
\end{lem}
\begin{proof} 
Multiplying Equation \eqref{eq1inserteq2} by $m \cdot \nabla \overline{u_n}$ and integrating over $\Omega_1$, and similarly for Equation \eqref{eq4inserteq5} over $\Omega_2$, we obtain
\begin{equation}
	- \int_{\Omega_1} \lambda_n^2  u_n \left(m \cdot \nabla \overline{u_n}\right) dx -\int_{\Omega_1} \Delta  u_n \left(m \cdot \nabla \overline{u_n}\right)dx= \int_{\Omega_1} \left(\frac{i f_n^1}{\lambda_n}+\frac{g_n^1}{\lambda_n^2}\right) \left(m \cdot \nabla \overline{u_n}\right) dx,\label{multmdotnablau}
	\end{equation}
	and
	\begin{equation}
	- \int_{\Omega_2} \lambda_n^2  w_n \left(m \cdot \nabla \overline{w_n}\right) dx +\int_{\Omega_2} \Delta^2 w_n \left(m \cdot \nabla \overline{w_n}\right)dx= \int_{\Omega_2} \left(\frac{i f_n^2}{\lambda_n}+\frac{g_n^2}{\lambda_n^2}\right) \left(m \cdot \nabla \overline{w_n}\right) dx.\label{multmdotnablaw}
	\end{equation}
	Using Cauchy-Schwartz inequality and the fact that $a(w_n,\overline{w_n})$ is equivalent to the usual norm of $H^2(\Omega_2)$ on $H_{\ast}^2 (\Omega_2)$, we can observe that $\nabla u_n$ is uniformly bounded in $L^2(\Omega_1)$ and $a(w_n,\overline{w_n})$ is uniformly bounded. Additionally, $f_n^1$ and $g_n^1$ converge to zero in $L^2(\Omega_1)$, while $f_n^2$ and $g_n^2$ converge to zero in $L^2(\Omega_2)$. Hence, we obtain
	\begin{equation}\label{rsideconvzero1}
	\left|\int_{\Omega_1} \left(\frac{i f_n^1}{\lambda_n}+\frac{g_n^1}{\lambda_n^2}\right) \left(m \cdot \nabla \overline{u_n}\right) dx\right| \leq \norm{m}_{\infty} \left(\frac{\norm{f_n^1}_{L^2(\Omega_1)}}{\lambda_n}+\frac{\norm{g_n^1}_{L^2(\Omega_1)}}{\lambda_n^2}\right)\norm{\nabla u_n}_{L^2(\Omega_1)}=\frac{o(1)}{\lambda_n},
	\end{equation}
and 	\begin{equation}\label{rsideconvzero2}
	\left|\int_{\Omega_2} \left(\frac{i f_n^2}{\lambda_n}+\frac{g_n^2}{\lambda_n^2}\right) \left(m \cdot \nabla \overline{w_n}\right) dx\right| \leq \norm{m}_{\infty} \left(\frac{\norm{f_n^2}_{L^2(\Omega_2)}}{\lambda_n}+\frac{\norm{g_n^2}_{L^2(\Omega_2)}}{\lambda_n^2}\right)\norm{\nabla w_n}_{L^2(\Omega_2)}=\frac{o(1)}{\lambda_n}.
	\end{equation}
	On the other hand, applying Green's formula to the first terms of \eqref{multmdotnablau} and \eqref{multmdotnablaw} yields
	\begin{equation}\label{firstterm1}
	\int_{\Omega_1} \lambda_n^2  u_n \left(m \cdot \nabla \overline{u_n}\right) dx = \frac{1}{2} \int_{\partial \Omega_1} \left( m \cdot \nu_1 \right) \left| \lambda_n u_n \right|^2 d\Gamma -\int_{\Omega_1} \left| \lambda_n u_n \right|^2 dx,
	\end{equation}
and 	\begin{equation}\label{firstterm2}
	\int_{\Omega_2} \lambda_n^2  w_n \left(m \cdot \nabla \overline{w_n}\right) dx = \frac{1}{2} \int_{\partial \Omega_2} \left( m \cdot \nu_2 \right) \left| \lambda_n w_n \right|^2 d\Gamma -\int_{\Omega_2} \left| \lambda_n w_n \right|^2 dx.
	\end{equation}
	Then, substituting \eqref{firstterm1} (resp. \eqref{firstterm2}) into \eqref{multmdotnablau} (resp. \eqref{multmdotnablaw}) and considering the estimates \eqref{rsideconvzero1} and \eqref{rsideconvzero2}, we obtain
	\begin{equation*}
	\int_{\Omega_1} \left| \lambda_n u_n \right|^2 dx -\frac{1}{2} \int_{\partial \Omega_1} \left( m \cdot \nu_1 \right) \left| \lambda_n u_n \right|^2 d\Gamma -\int_{\Omega_1} \Delta  u_n \left(m \cdot \nabla \overline{u_n}\right)dx= \frac{o(1)}{\lambda_n},
	\end{equation*}
	and
	\begin{equation*}
	\int_{\Omega_2} \left| \lambda_n w_n \right|^2 dx - \frac{1}{2} \int_{\partial \Omega_2} \left( m \cdot \nu_2 \right) \left| \lambda_n w_n \right|^2 d\Gamma +\int_{\Omega_2} \Delta^2 w_n \left(m \cdot \nabla \overline{w_n}\right)dx= \frac{o(1)}{\lambda_n}.
	\end{equation*}
	Adding both equations and using the boundary condition $u_n=w_n$ on the interface $I$ and the estimates \eqref{dissconseq1} and \eqref{dissconseq3}, then taking the real part, we obtain
	\begin{equation}\label{addproof}
\int_{\Omega_1} \left| \lambda_n u_n \right|^2 dx + \int_{\Omega_2} \left| \lambda_n w_n \right|^2 dx +\Re\left\{\int_{\Omega_2} \Delta^2 w_n (m \cdot \nabla \overline{w_n}) dx -\int_{\Omega_1} \Delta u_n (m \cdot \nabla \overline{u_n}) dx\right\}=o(1).
	\end{equation}
	As Assumptions \ref{anglesomega1}, \ref{anglesomega2} and \ref{geometriccon} hold, it suffices to apply Lemma \ref{lemu&winequality}. Hence, inserting \eqref{u&winequality} into \eqref{addproof}, we get
	\begin{equation}\begin{split}\label{addproof&lemma}
\int_{\Omega_1} \left| \lambda_n u_n \right|^2 dx + \int_{\Omega_2} \left| \lambda_n w_n \right|^2 dx+\frac{1}{2} a(w_n,\overline{w_n})-\frac{R_1^2}{\delta} \int_{\Gamma_1} |\partial_{\nu_1} u_n|^2 d\Gamma  -\frac{\delta(1-\mu)}{4R_2^2}\int_{\Gamma_2} |\partial_{\nu_2} w_n|^2  d\Gamma  \\-\frac{2R_2^2}{\delta(1-\mu)}\int_{\Gamma_2}|\mathcal{B}_1 w_n|^2d\Gamma -\frac{R_2^2 M}{2}\int_{\Gamma_2}|\mathcal{B}_2 w_n|^2d\Gamma =o(1).
	\end{split}\end{equation}
	Finally, using the estimations \eqref{dissconseq2}, \eqref{domainconseq1}, \eqref{domainconseq2} and \eqref{domainconseq3}, it follows from \eqref{addproof&lemma} that 
	\begin{equation*}
\int_{\Omega_1} \left| \lambda_n u_n \right|^2 dx + \int_{\Omega_2} \left| \lambda_n w_n \right|^2 dx+\frac{1}{2} a(w_n,\overline{w_n})=o(1),
	\end{equation*}
	which proves the estimation \eqref{u&w&aww}, as desired.
\end{proof}
\\ \\ \textbf{Proof of Theorem \ref{PolynomialStability}.}  Using \eqref{u&wequiv} and \eqref{u&w&aww}, we deduce that
\begin{equation}\label{estnablau}
\int_{\Omega_1} |\nabla u_n|^2 dx =o(1).
\end{equation}
Therefore, combining the estimates \eqref{esdiss}, \eqref{u&w&aww} and \eqref{estnablau}, we obtain $\norm{U_n}_{\HH} \longrightarrow 0,$ which leads to the desired contradiction with \eqref{c1}. Consequently, condition $(P2)$ holds and this permits us to conclude that the energy of system \eqref{p5-sys2}-\eqref{p5-diff} decays polynomially to zero as $t$ goes to infinity. The proof is thus complete.
\section{Conclusion and Open Problems}

In this paper, we present a study of the stabilization of a transmission wave-plate model coupled through the interface with dynamical boundary controls. By employing a general criterion proposed by Arendt-Batty, we have successfully demonstrated the strong stability of the system. Additionally, we have proven the lack of exponential stability. Notably, under certain geometric assumptions on the boundary, particularly when the interface between the wave and the plate is straight, we establish a polynomial decay rate of the energy of type $1/t$.
\\ 

An intriguing question arising from the problem studied in this paper is the optimality of the polynomial decay rate through spectral analysis. Furthermore, it would be interesting to extend the obtained results to the case of a wave-plate model subject to only one or two dynamical boundary feedbacks at the exterior boundaries of the wave's and the plate's domains. Such generalizations could lead to further insights and applications in related fields.

\section*{Acknowledgments}
Sincere thanks to the anonymous referees and the associate editor for their valuable comments and useful suggestions. Mr. Ali Wehbe would like to thank the LAMA Laboratory of Mathematics at the Université Savoie Mont Blanc for their support in the research development project. Ms. Zahraa Abdallah is grateful to the Doctoral School of Science and Technology at the Lebanese University for their financial support through the doctoral scholarship. We also extend our appreciation for the kind hospitality at the LAMA Laboratory to Ms. Zahraa Abdallah during her research stays at the USMB.

\bibliographystyle{siam}

\end{document}